\documentclass[reqno]{amsart}
\usepackage{amssymb}
\usepackage{amsfonts}
\usepackage{amsmath}
\usepackage{color}
\usepackage{amsaddr}

\newtheorem{theorem}{Theorem}[section]

\newtheorem{corollary}[theorem]{Corollary}

\newtheorem{definition}[theorem]{Definition}

\newtheorem{lemma}[theorem]{Lemma}

\newtheorem{proposition}[theorem]{Proposition}

\theoremstyle{remark}
\newtheorem{remark}[theorem]{Remark}

\numberwithin{equation}{section}

\newcommand{\ep}{\varepsilon}
\newcommand{\vp}{\varphi}

\begin{document}

\title[Viscous Fractional Cahn-Hilliard Equations with Memory]{Well-posedness and Global Attractors for Viscous Fractional Cahn-Hilliard Equations with Memory}

\author[E. \"{O}zt\"{u}rk and J. L. Shomberg]{Eylem \"{O}zt\"{u}rk$^1$ and Joseph L. Shomberg$^2$}



\address{$^1$Department of Mathematics, Hacettepe University, 06800 Beytepe, Ankara, Turkey, \\ {\tt{eozturk@hacettepe.edu.tr}}}

\address{$^2$Department of Mathematics and Computer Science, Providence College, Providence, Rhode Island 02918, USA, \\ {\tt{jshomber@providence.edu}
}}

\date{\today }

\begin{abstract}
We examine a viscous Cahn-Hilliard phase-separation model with memory and where the chemical potential possesses a nonlocal fractional Laplacian operator.
The existence of global weak solutions is proven using a Galerkin approximation scheme. 
A continuous dependence estimate provides uniqueness of the weak solutions and also serves to define a precompact pseudometric.
This, in addition to the existence of a bounded absorbing set, shows that the associated semigroup of solution operators admits a compact connected global attractor in the weak energy phase space.
The minimal assumptions on the nonlinear potential allow for arbitrary polynomial growth.
\end{abstract}

\maketitle
\tableofcontents

\section{Introduction}

Let $\Omega$ be a smooth (at least Lipschitz) bounded domain in $\mathbb{R}^{N}$, $N=3,2,1$, with boundary $\partial\Omega$ and let $T>0$.
We consider the following viscous fractional Cahn-Hilliard equation in the unknown (order parameter) $u$ satisfying
\begin{align}
\partial_tu(t,x)=\int_0^\infty k(s)\Delta\mu(t-s,x) ds \quad \text{in} \quad \Omega\times(0,T),  \label{pde1}
\end{align}
$k$ is a so-called relaxation kernel, with a chemical potential $\mu$ given by
\begin{align}
\mu(t,x)=\alpha \partial_tu(t,x)+(-\Delta)^\beta u(t,x) + F'(u(t,x)) \quad \text{in} \quad \Omega\times\mathbb{R},  \label{pde2}
\end{align}
$\alpha>0,$ $\beta\in(0,1)$, and typically $F$ is a double-well potential (the precise assumptions on $F$ are stated in (N1)-(N3) below), subject to the boundary conditions
\begin{align}
u=0 \quad \text{on} \quad \mathbb{R}^{N}\backslash\Omega\times(0,T) \quad \text{and} \quad \partial_{\bf{n}}\mu=0 \quad \text{on} \quad \partial\Omega\times(0,T),  \label{bc}
\end{align}
with the given initial and past conditions
\begin{align}
u(0)=u_0(0) \quad \text{in} \quad \Omega \quad \text{and} \quad u(-t)=u_{0}(-t) \quad \text{in} \quad \Omega\times[0,T),  \label{ic}
\end{align}
for 
\[
u_0:\Omega\times(-\infty,0)\rightarrow\mathbb{R}.
\]

Here we define $(-\Delta)^\beta$ with $0<\beta<1$ as the (nonlocal) fractional Laplace operator.
In other words, let $\Omega \subset {\mathbb{R}}^{N}$ be an arbitrary open set and fix
\begin{equation*}
\mathcal{L}^{1}(\Omega) := \left\{ u:\ \Omega \rightarrow {\mathbb{R}}\
\mbox{measurable,}\ \int_{\Omega} \frac{|u(x)|}{(1+|x|)^{N+2\beta}} dx < \infty \right\}.
\end{equation*}
For $u\in \mathcal{L}^{1}({\mathbb{R}}^{N})$, $x\in {\mathbb{R}}^{N}$ and $\varepsilon >0$, we write
\begin{equation*}
(-\Delta)_{\varepsilon }^\beta u(x)=C_{N,\beta}\int_{\{y\in {\mathbb{R}}^N,|y-x|>\varepsilon \}}\frac{u(x)-u(y)}{|x-y|^{N+2\beta}} dy
\end{equation*}
with the normalized constant $C_{N,\beta}$ given by
\begin{align}
C_{N,\beta}=\frac{\beta2^{2\beta}\Gamma \left(\frac{N+2\beta}{2}\right)}{\pi^{\frac{N}{2}}\Gamma(1-\beta)},  \label{constant}
\end{align}
where $\Gamma$ denotes the usual Gamma function.
The (restricted) fractional Laplacian $(-\Delta)^\beta u$ of the function $u$ is defined by the formula
\begin{align}
(-\Delta)^\beta u(x) = C_{N,\beta}\mbox{P.V.}\int_{{\mathbb{R}}^N}\frac{u(x)-u(y)}{|x-y|^{N+2\beta}} dy = \lim_{\varepsilon \downarrow 0}(-\Delta)_\varepsilon^\beta u(x),\quad x\in{\mathbb{R}}^N,  \label{eq11}
\end{align}
provided that the limit exists.
We call $A^\beta$ the self-adjoint realization of the fractional Laplacian $(-\Delta)^\beta$ with Dirichlet boundary condition (\ref{bc})$_1$, see, e.g., \cite[Section 2.2]{Gal-Warma-15-1} (cf. also \cite{Warma15}).

Some remarks: First, observe the chemical potential \eqref{pde2} involves the Neumann (no-flux) condition described by \eqref{bc}$_2$.
Hence, when the memory function $k$ is {\em close} to the Dirac delta function, we recover the usual parabolic equation associated with the Cahn-Hilliard equation with the flux-free chemical potential.

Naturally, we are also interested in the closely related problem to \eqref{pde1}-\eqref{ic} whereby the fractional Laplace operator $(-\Delta)^\beta$ is replaced with the {\em{regional}} fractional Laplacian, $A^\beta_\Omega$, defined by first setting
\begin{equation*}
A^\beta_{\Omega,\varepsilon}u(x)=C_{N,\beta}\int_{\{y\in\Omega,|y-x|>\varepsilon \}}\frac{u(x)-u(y)}{|x-y|^{N+2\beta}} dy,
\end{equation*}
where $C_{N,\beta}$ is given by \eqref{constant}, then
\begin{align}
A^\beta_{\Omega}u(x) = C_{N,\beta}\mbox{P.V.}\int_{\Omega}\frac{u(x)-u(y)}{|x-y|^{N+2\beta}} dy = \lim_{\varepsilon \downarrow 0}A^\beta_{\Omega,\varepsilon}u(x),\quad x\in\Omega,  \label{eq12}
\end{align}
provided that the limit exists.
Assuming $u\in\mathcal{D}(\Omega)$ (cf. \cite[page 1280]{Gal-Warma-15-1}) then the two fractional Laplacian operators are related by 
\begin{align}
(-\Delta)^\beta u(x)=A^\beta_\Omega u(x)+V_\Omega(x)u(x), \quad \forall u \in \mathcal{D}(\Omega)  \label{Lap-id}
\end{align}
with the following potential
\begin{align}
V_\Omega(x) := C_{N,\beta}\int_{\mathbb{R}^N\setminus\Omega} \frac{dy}{|x-y|^{N+2\beta}}, \quad x\in\Omega.  \label{G-pot}
\end{align}
The comparable Cahn-Hilliard problem with the regional fractional Laplacian is then \eqref{pde1} with the chemical potential
\begin{align}
\mu=\alpha \partial_tu+A^\beta_\Omega u+F'(u) \quad \text{in} \quad \Omega\times(0,T),  \label{pde2-2}
\end{align}
now subject to the boundary conditions
\begin{align}
u=0 \quad \text{on} \quad \partial\Omega \times(0,T) \quad \text{and} \quad \partial_{\bf{n}}\mu=0 \quad \text{on} \quad \partial\Omega\times(0,T),  \label{bc-2}
\end{align}
with the above initial and past conditions in \eqref{ic}.
Our focus here is on obtaining results for the restricted fractional Laplacian, of which the regional counterpart can be view as a perturbation thanks to \eqref{Lap-id}.
The restricted fractional Laplacian appears in the context of nonlocal phase transitions with Dirichlet boundary conditions in \cite{Bucur-Valdinoci-16,CRS-10}.
On the other hand, the regional fractional Laplacian is generally better suited to treat problems with nonhomogeneous boundary data and even dynamic boundary conditions (see \cite{Gal-Warma-15-1,Gal-NCHEFDBC-16} and the references therein).

Inside a bounded container $\Omega\subset\mathbb{R}^3,$ the Cahn-Hilliard equation (cf. \cite{Cahn-Hilliard}) is a phase separation model for a binary solution (e.g. a cooling alloy, glass, or polymer),
\[
\partial_tu = \nabla\cdot[\kappa(u)\nabla\mu],
\]
where $u$ is the {\em{order-parameter}} (the relative difference of the two phases), $\kappa$ is the {\em{mobility function}} (which we set $\kappa\equiv1$ throughout this article), and $\mu$ is the {\em{chemical potential}} (the first variation of the free-energy $E$ with respect to $u$).
In the classical model,
\[
\mu = -\Delta u + F'(u) \quad \text{and} \quad E(u) = \int_\Omega \left( \frac{1}{2}|\nabla u|^2 + F(u) \right) dx,
\]
where $F$ describes the density of potential energy in $\Omega$ (e.g. the double-well potential $F(s)=(1-s^2)^2$).

Recently the nonlocal free-energy functional appears in the literature \cite{Giacomin-Lebowitz-97},
\[
E(\phi) = \int_\Omega\int_\Omega \frac{1}{4}J(x-y)(\phi(x)-\phi(y))^2 dxdy + \int_\Omega F(\phi) dx,
\]
hence, the {\em{chemical potential}} is, 
\begin{align}
\mu = a\phi - J*\phi + F'(\phi), \label{wk-non1}
\end{align}
where
\begin{align}
a(x) = \int_\Omega J(x-y) dy \quad \text{and} \quad (J*\phi)(x) = \int_\Omega J(x-y)\phi(y) dy. \label{wk-non2}
\end{align}
In view of \cite{Gal-16-doubly,Gal-17-str-str-CHE}, the nonlocality expressed in \eqref{wk-non1}-\eqref{wk-non2} (cf. also \cite{AVMRTM10,Bates&Han05,CFG12,Frigeri&Grasselli12,Frigeri-Grasselli-Rocca_2014,Gal&Grasselli14,Grasselli-Schimperna-2011,Porta-Grasselli-2014,Shomberg-n16,Shomberg18-1}) is termed {\em weak} while the type under consideration here in \eqref{pde2} and \eqref{eq11} is called {\em strong}.
Under certain conditions the strong type reduces to the weak (cf. \cite{Gal-16-doubly}, and also see \cite{Giacomin-Lebowitz-97}).
Recently there has been much interest in the nonlocal Cahn--Hilliard equation with strong interactions of the restricted fractional Laplacian type \eqref{eq11} and the regional fractional Laplacian type \eqref{eq12} (cf. \cite{GA-GS-AS_2015,Bucur-Valdinoci-16,Gal-16-doubly,Gal-NCHEFDBC-16,Gal-17-str-str-CHE}).
The results in these references concern global well-posedness, and when available, the existence of finite dimensional global attractors and regularity.

Additionally, there has been exceptional growth concerning dissipative infinite-dimensional systems with {\em memory} including models arising in the theory of heat conduction in special materials (cf. e.g. \cite{CPS06,Gal-Shomberg15-2,GPM98,GPM00,Shomberg-reacg16}) and the
theory of phase-transitions (cf. e.g. \cite{CGG11,Conti-Mola-08,Gal&Grasselli12,GGMP05-mem,GGPS05,GGP01,Grasselli08,Grasselli&Pata05,GPV03}). 
One feature of equations that undergo ``memory relaxation'' is admissibility of a so-called {\em inertia} term.
For example (cf. e.g. \cite{GMPZ10}) the first-order equation with memory
\[
u_t(t) + \int_0^\infty k_\ep(s) f(u(t-s)) ds = 0 
\]
for 
\[
k_\ep(s) = \frac{1}{\ep} e^{-s/\ep}
\]
leads us (formally) to the ``hyperbolic relaxation'' equation
\[
\ep u_{tt}(t) + u_t(t) + f(u(t)) = 0.
\]
In this way, our model also includes the Cahn--Hilliard equation with inertial term (cf. \cite{CGW-14,GPS07,GSSZ-09,SS-16}). 
Hence, the novelty in the present work is a relaxation of a phase-field model with a strongly interacting nonlocal diffusion mechanism.  

In this article, our aims are:
\begin{itemize}
\item To provide a framework to establish the global (in time) well-posedness of the model problems \eqref{pde1}-\eqref{ic} and \eqref{pde1}, \eqref{ic}, \eqref{pde2-2} and \eqref{bc-2}.

\item To prove the semigroup of solution operators admits a compact global attractor.
\end{itemize}

In order to reach these aims, we require sufficient growth conditions on $F$ (given below) in order to employ a Galerkin scheme with suitable {\em{a priori}} estimates.
With a finite energy phase space identified, a one-parameter family of solution operators is defined, hence generating a semi-dynamical system.
This semigroup is dissipative on the energy phase-space and also defines an $\alpha$-contraction on the phase-space.
The existence of a compact global attractor follows.

\section{Past history formulation and functional setup}  \label{s:functional}

We now introduce the well-established past history approach from \cite{Grasselli&Pata02-2} (cf. also \cite{Conti-Mola-08,GGMP05-mem}) by defining the past history variable, for all $s>0$ and $t>0,$
\begin{align}  \label{eta-1}
\eta^t(x,s)=\int_0^s -\Delta\mu(x,t-\sigma)d\sigma.
\end{align}
Observe, $\eta$ satisfies the boundary condition
\begin{align}  \label{eta-0}
\eta^t(x,0)=0 \quad \text{on} \quad \Omega\times(0,\infty).
\end{align}
When $k$ is sufficiently smooth and vanishes at $+\infty$ (these assumptions will be made more precise  below), then integration by parts yields
\[
\int_0^\infty k(s)\Delta\mu(x,t-s)ds = -\int_0^\infty \nu(s)\eta^t(x,s)ds
\]
where $\nu(s)=-k'(s)$.

We may now formulate the model problem \eqref{pde1}-\eqref{ic} as:

\noindent Problem {\textbf P}. 
Find $(u,\eta)=(u(x,t),\eta^t(x,s))$ on $(0,\infty)$ such that
\begin{eqnarray}
\partial_tu(x,t)+\int_0^\infty \nu(s)\eta^t(x,s)ds=0 & \text{in} & \Omega\times(0,\infty)  \label{p-1} \\
\mu(x,t) = \alpha\partial_t u(x,t) + (-\Delta)^\beta u(x,t) + F'(u(x,t)) & \text{in} & \Omega\times(0,\infty)  \label{p-2} \\
\partial_t\eta^t(x,s) + \partial_s\eta^t(x,s) = -\Delta\mu(x,t) & \text{in} & \Omega\times(0,\infty)\times(0,\infty)  \label{p-3}
\end{eqnarray}
hold subject to \eqref{bc} and \eqref{eta-0}, and satisfying the initial conditions \eqref{ic}$_1$ and
\begin{eqnarray}
\eta^0(x,s)=\eta_0(x,s) \quad \text{in} \quad \Omega\times(0,\infty),  \label{p-5}
\end{eqnarray}
whereby with \eqref{eta-1},
\begin{align}  \label{p-6}
\eta_0(x,s)=\int_0^s-\Delta \mu_0(x,-y)dy \quad \text{in} \quad \Omega\times(0,\infty),
\end{align}
where in light of \eqref{ic}$_2$,
\begin{align}  \label{mu-ic}
\mu_0(x,t)=\alpha\partial_tu_0(x,t)+(-\Delta)^\beta u_0(x,t) + F'(u_0(x,t)) \quad \text{for} \quad t\le0.
\end{align}
Additionally, we are also interested in treating the related problem where the above fractional Laplace operator $(-\Delta)^\beta$ is replaced with the regional counterpart $A^\beta_\Omega$.
Hence, the formulation of the related {\em regional} Problem {\textbf P} is based on \eqref{pde1}, \eqref{ic}, \eqref{pde2-2}, and \eqref{bc-2}.

Here we introduce some notation.
From now on, we denote by $\Vert \cdot \Vert _X$, the norm in the specified (real) Banach space $X$, and $(\cdot,\cdot)_{Y}$ denotes the product on the specified (real) Hilbert space $Y$.
The dual pairing between $Y$ and the dual $Y^*$ is denoted by $\langle u,v\rangle_{Y^*\times Y}$.
The set $\Omega$ is omitted from the space when we indicate the norm.
We denote the measure of the domain $\Omega $ by $|\Omega|$.
In many calculations, functional notation indicating dependence on the variable $t$ is dropped; for example, we will write $u$ in place of $u(t)$ or $\eta^t$ in place of $\eta^t(s)$.
Throughout the paper, $C$ will denote a \emph{generic} positive constant, while $Q:\mathbb{R}_{+}\rightarrow \mathbb{R}_{+}$ will denote a \emph{generic} increasing function.
Such generic terms may or may not indicate dependencies on the (physical) parameters of the model problem, and may even change from line to line.

Let us define the linear operator $A_N:=-\Delta$ on $D(A_N)=\{\psi\in H^2(\Omega):\partial_{\bf{n}}\psi=0\ \text{on}\ \partial\Omega\}$, as the realization in $L^2(\Omega)$ of the Laplace operator endowed with Neumann boundary conditions.
Here, $-\Delta$ denotes the usual (local) Laplace operator.
It is well-known that $A_N$ is the generator of a bounded analytic semigroup $e^{-A_Nt}$ on $L^2(\Omega)$.
Additionally, $A_N$ is nonnegative and self-adjoint on $L^2(\Omega).$
With $H^{-r}(\Omega):=(H^{r}(\Omega))^*$, $r\in\mathbb{N}_+$, denote by $\langle\cdot\rangle$ the spatial average over $\Omega;$ i.e.,
\[
\langle \psi\rangle:=\frac{1}{|\Omega|}\langle \psi,1\rangle_{H^{-r}\times H^r}.
\]
We set $H^r_{(0)}(\Omega)=\{\psi\in H^r(\Omega):\langle\psi\rangle=0\}$, $H^0(\Omega)=L^2(\Omega),$ and we know that $A^{-1}_N:H^0_{(0)}(\Omega)\rightarrow H^0_{(0)}(\Omega)$ is a well-defined mapping.
We will refer to the following norms in $H^{-r}(\Omega)$ (which are equivalent to the usual norms)
\begin{equation}
\|\psi\|^2_{H^{-r}}=\|A_N^{-r/2}(\psi-\langle\psi\rangle)\|^2+|\langle\psi \rangle|^2.  \label{H-neg}
\end{equation}
The Sobolev space $H^1(\Omega)$ is endowed with the norm,
\begin{equation}  \label{H1-norm}
\|\psi\|^2_{H^1}:= \|\nabla\psi\|^2 + \langle \psi \rangle^2.
\end{equation}
Denote by $\lambda_\Omega>0$ the constant in the Poincar\'{e}-Wirtinger inequality,
\begin{equation}  \label{Poincare}
\|\psi-\langle\psi\rangle\| \le \sqrt{\lambda_\Omega}\|\nabla\psi\|.
\end{equation}
Whence, for $\lambda^*_\Omega:=\max\{\lambda_\Omega,1\}$, there holds, for all $\psi\in H^1(\Omega),$
\begin{align}
\|\psi\|^2 & \le \lambda_\Omega\|\nabla\psi\|^2 + \langle\psi\rangle^2  \label{Poincare2} \\
& \le \lambda^*_\Omega\|\psi\|^2_{H^1}.  \notag
\end{align}

We now more rigorously describe the fractional Laplacian with Dirichlet boundary conditions.
For an arbitrary bounded domain $\Omega\subset\mathbb{R}^N$ and for $\beta\in(0,1)$, denote the fractional-order Sobolev space by,
\begin{equation*}
W^{\beta,2}(\Omega) := \left\{ u\in L^2(\Omega) : \int_\Omega \int_\Omega \frac{|u(x)-u(y)|^2}{|x-y|^{N+2\beta}} dxdy < \infty \right\},
\end{equation*}
to be equipped with the norm
\begin{equation*}
\|u\|_{W^{\beta,2}} := \left( \int_\Omega |u(x)|^2dx + \frac{C_{N,\beta}}{2} \int_\Omega\int_\Omega \frac{|u(x)-u(y)|^2}{|x-y|^{N+2\beta}} dxdy \right)^{1/2},
\end{equation*}
where $C_{N,\beta}$ is given by (\ref{constant}).
Let
\begin{equation*}
W_0^{\beta,2}(\Omega) = {\overline{\mathcal{D}(\Omega)}}^{W^{\beta,2}(\Omega)}.
\end{equation*}
Hence, $W_0^{\beta,2}(\Omega)$ is a closed subspace of $W^{\beta,2}(\Omega)$ containing $\mathcal{D}(\Omega)$.
Moreover, thanks to \cite[Theorem 10.1.1]{Adams-Hedberg-96},
\begin{equation*}
W_0^{\beta,2}(\Omega) = \{ u\in W^{\beta,2}(\mathbb{R}^N) : \tilde u=0 \ \text{on}\ \mathbb{R}^N \setminus \Omega \},
\end{equation*}
where $\tilde u$ is the quasi-continuous version (with respect to the capacity defined with the space $W^{\beta,2}(\Omega)$) of $u$.
One may easily show that the following defines an equivalent norm on the space $W_0^{\beta,2}(\Omega)$,
\begin{align}
|\|u\||^2_{W_0^{\beta,2}} & = \frac{C_{N,\beta}}{2} \int_\Omega \int_\Omega \frac{|u(x)-u(y)|^2}{|x-y|^{N+2\beta}} dxdy + \int_\Omega  V_\Omega(x) |u(x)|^2 dx  \notag \\
& = \frac{C_{N,\beta}}{2} \int_{\mathbb{R}^N} \int_{\mathbb{R}^N} \frac{|u(x)-u(y)|^2}{|x-y|^{N+2\beta}} dxdy.  \label{equiv-norm}
\end{align}
Here, $V_\Omega$ is the potential \eqref{G-pot}.

\begin{remark}
Either definition of the space $W^{\beta,2}_0(\Omega)$ makes sense for any arbitrary open set $\Omega\subset\mathbb{R}^3$ (not necessarily bounded).
Also, if $\Omega$ has Lipschitz boundary, then by \cite{BBC-03}, $W^{\beta,2}_0(\Omega)=W^{\beta,2}(\Omega)$ for every $0<\beta\le\frac{1}{2}$.
\end{remark}

From now on, we write $u\in W_0^{\beta,2}(\Omega)$ to mean $u\in W^{\beta,2}(\mathbb{R}^N)$ and $u=0$ on $\mathbb{R}^N\setminus\Omega.$
Let $a_{E,\beta}$ be the bilinear symmetric closed form with domain $D(a_{E,\beta})=W_0^{\beta,2}(\Omega)$ and defined for $u,v\in W_0^{\beta,2}(\Omega)$ by
\begin{align}
a_{E,\beta}(u,v) & = \frac{C_{N,\beta}}{2} \int_\Omega \int_\Omega \frac{(u(x)-u(y))(v(x)-v(y))}{|x-y|^{N+2\beta}}dxdy + \int_\Omega V_\Omega(x)u(x)v(x) dx  \notag \\
& = \frac{C_{N,\beta}}{2} \int_{\mathbb{R}^N} \int_{\mathbb{R}^N} \frac{(u(x)-u(y))(v(x)-v(y))}{|x-y|^{N+2\beta}}dxdy.  \label{form-1}
\end{align}
Let $A_{E,\beta}$ be the closed linear self-adjoint operator on $L^2(\Omega)$ associated with $a_{E,\beta}$ by
\begin{align}  \label{op-1}
\left\{ \begin{array}{l} D(A_{E,\beta}):=\{ u\in W^{\beta,2}_0(\Omega): \exists v\in L^2(\Omega),\ a_{E,\beta}(u,\vp)=(v,\vp)\ \forall\vp\in W^{\beta,2}_0(\Omega) \}  \\
A_{E,\beta}u=v. \end{array} \right.
\end{align}
According to \cite[Proposition 2.2]{Gal-Warma-15-1}, the operator $A_{E,\beta}$ on $L^2(\Omega)$ associated with the bilinear form $a_{E,\beta}$ is given by
\begin{align}  \label{op-2}
D(A_{E,\beta}):=\{u\in W_0^{\beta,2}(\Omega) : (-\Delta)_E^\beta u\in L^2(\Omega) \} \quad \text{and} \quad \forall u\in D(A_{E,\beta}), \quad A_{E,\beta}u:=(-\Delta)_E^\beta u.
\end{align}
Observe, comparing \eqref{eq11} and (\ref{equiv-norm})-(\ref{op-2}) shows, for all $u\in D(A_{E,\beta})$,
\begin{align}  \label{green-light}
((-\Delta)_E^\beta u,u) & = a_{E,\beta}(u,u) = |\|u\||^2_{W^{\beta,2}_0}.
\end{align}

Concerning the related regional problem discussed above, we let $a_{D,\beta}$ be the bilinear symmetric closed form with domain $D(a_{D,\beta})=W_0^{\beta,2}(\Omega)$ and defined for $u,v\in W_0^{\beta,2}(\Omega)$ by
\begin{align}
a_{D,\beta}(u,v) & = \frac{C_{N,\beta}}{2} \int_\Omega \int_\Omega \frac{(u(x)-u(y))(v(x)-v(y))}{|x-y|^{N+2\beta}}dxdy.  \label{form-2}
\end{align}
Let $A_{D,\beta}$ be the closed linear self-adjoint operator on $L^2(\Omega)$ associated with $a_{D,\beta}$ by
\begin{align}
\left\{ \begin{array}{l} D(A_{D,\beta}):=\{ u\in W^{\beta,2}_0(\Omega): \exists v\in L^2(\Omega),\ a_{D,\beta}(u,\vp)=(v,\vp)\ \forall\vp\in W^{\beta,2}_0(\Omega) \}  \\
A_{D,\beta}u=v. \end{array} \right.  \label{rop-1}
\end{align}
Then by \cite[Proposition 2.3]{Gal-Warma-15-1}, the operator $A_{D,\beta}$ on $L^2(\Omega)$ associated with the bilinear form $a_{D,\beta}$ is given by
\begin{align}
D(A_{D,\beta}):=\{u\in W_0^{\beta,2}(\Omega) : A^\beta_\Omega u\in L^2(\Omega)\} \quad \text{and} \quad \forall u\in D(A_{D,\beta}), \quad A_{D,\beta}u:=A^\beta_\Omega u.  \label{rop-2}
\end{align}

We introduce the spaces for the memory variable $\eta.$
First, the product in $H^\sigma(\Omega)$ for $\sigma\in\mathbb{R}$ and $u_1,u_2\in H^\sigma(\Omega)$ is defined by
\begin{align} \label{H-prod}
(u_1,u_2)_{H^\sigma}=(A_N^{\sigma/2}u_1,A_N^{\sigma/2}u_2).
\end{align}
For a nonnegative measurable function $\theta$ defined on $\mathbb{R}_+$ and for a Hilbert space $W$ (with inner-product $(\cdot,\cdot)_W$), let $L^2_\theta(\mathbb{R}_+;W)$ be the Hilbert space of $W$-valued functions on $\mathbb{R}_+$ equipped with the following product,
\[
(\phi_1,\phi_2)_{L^2_\theta(\mathbb{R}_+;W)}=\int_0^\infty \theta(s)(\phi_1(s),\phi_2(s))_W ds.
\]
Thus, we set
\[
\mathcal{M}_\sigma=L^2_\nu(\mathbb{R}_+;H^\sigma(\Omega)) \quad \text{and} \quad \mathcal{M}^{(0)}_\sigma=L^2_\nu(\mathbb{R}_+;H_{(0)}^\sigma(\Omega)) \quad \text{for $\sigma\in\mathbb{R}$},
\]
where $\nu=\nu(s)$ is the kernel from \eqref{p-1}.
Hence, for $\sigma\in\mathbb{R}$ and $\phi_1,\phi_2\in\mathcal{M}_\sigma$, using \eqref{H-prod} the product in $\mathcal{M}_{\sigma}$ (and $\mathcal{M}_{\sigma}^{(0)}$) can be expressed as
\[
(\phi_1,\phi_2)_{\mathcal{M}_\sigma}=\int_0^\infty \nu(s)(A_N^{\sigma/2}\phi_1(s),A_N^{\sigma/2}\phi_2(s)) ds.
\]
Naturally, we may also consider spaces of the form $H_\nu^k(\mathbb{R}_+;H^\sigma(\Omega))$ for $k\in\mathbb{N}$.

We mention that solutions of Problem {\textbf P} must also satisfy the mass conservation constraints,
\begin{equation}
\langle u(t)\rangle = \langle u_0(0)\rangle \quad \text{and} \quad \langle\eta^t(s)\rangle = 0 \quad \forall t>0,\ \forall s>0.  \label{con}
\end{equation}
With this, it is important to realize that the norm of $\eta^t$ in the space $\mathcal{M}_{-1}^{(0)}$ may be expressed {\em without} writing the average value of $\eta_0$ in \eqref{H-neg} by virtue of the second of \eqref{con}.
Indeed, for $\eta^t\in\mathcal{M}_{-1}^{(0)}$,
\begin{align*}
\|\eta^t\|_{\mathcal{M}_{-1}} & = \left( \int_0^\infty \nu(s)\|\eta^t(s)\|^2_{H^{-1}} ds \right)^{1/2}  \notag \\
& = \left( \int_0^\infty \nu(s)\|A_N^{-1/2}\eta^t(s)\|^2 ds \right)^{1/2}.  \notag
\end{align*}

We now state the basic function spaces we intend to study Problem {\textbf P} in.
For each $\beta\in(0,1)$ and $\sigma\in\mathbb{R}$, define the following (weak) energy Hilbertian phase-space $\mathcal{H}_{\beta,\sigma}:=W^{\beta,2}_0(\Omega)\times\mathcal{M}_{\sigma-1}^{(0)}$, equipped with the norm on $W^{\beta,2}_0(\Omega)\times\mathcal{M}_{\sigma-1}^{(0)}$ whose square is given by, for all $\phi=(u,\eta)^{tr}\in\mathcal{H}_{\beta,\sigma}$,
\begin{align}
\|\phi\|^2_{\mathcal{H}_{\beta,\sigma}} & := \|u\|^2_{W^{\beta,2}_0} + \|\eta^t\|^2_{\mathcal{M}_{\sigma-1}}.  \notag
\end{align}
Then, for each $M\ge0$, define the closed subset
\begin{align}
\mathcal{H}^M_{\beta,\sigma} = \{ \phi=(u,\eta)^{tr}\in \mathcal{H}_{\beta,\sigma} : |\langle u\rangle|\le M \}.  \label{space}
\end{align}
When we are concerned with the dynamical system associated with the model Problem {\textbf P}, we will utilize the following metric space,
\begin{align}
\mathcal{X}^M_{\beta,\sigma} := \left\{ \phi=(u,\eta)^{tr}\in\mathcal{H}^M_{\beta,\sigma} : F(u)\in L^1(\Omega) \right\},  \notag
\end{align}
endowed with the metric
\begin{align}
d_{\mathcal{X}^M_{\beta,\sigma}}(\phi_1,\phi_2) := \|\phi_1-\phi_2\|_{\mathcal{H}^M_{\beta,\sigma}} + \left| \int_\Omega F(u_1)dx - \int_\Omega F(u_2)dx \right|^{1/2}.  \notag
\end{align}

\begin{remark}
The embedding $\mathcal{H}^M_{\beta,1} \hookrightarrow \mathcal{H}^M_{\beta,0}$ is continuous but not compact, due to the presence of the second component $\mathcal{M}_{\sigma-1}^{(0)}$.
Indeed, see \cite{Pata-Zucchi-2001} for a counterexample.
\end{remark}

It is appropriate for us to state the various assumptions that may used on the kernel $\nu$.

\begin{description}
\item[(K1)] $\nu\in C^1(\mathbb{R}_+)\cap L^1(\mathbb{R}_+)$ and $\nu(s)\ge0$ for all $s\in\mathbb{R}_+$.
\item[(K2)] $\nu'(s)\le0$ for all $s\in\mathbb{R}_+$.
\item[(K3)] $k_0=\displaystyle\int_0^\infty \nu(s)ds>0$. (For the sake of simplicity we now assume $k_0=1$ throughout the rest of the paper.)
\item[(K4)] $\nu_0=\displaystyle\lim_{s\rightarrow0^+}\nu(s)<\infty.$
\item[(K5)] $\nu'(s)+\lambda\nu(s)\le0$ for a.a. $s\in\mathbb{R}_+$, for some $\lambda>0.$
\end{description}

Some remarks for these assumptions. By assumption (K2), the inequality holds for all $\eta^t\in D(T_r)$
\begin{equation}
(T_r\eta^t,\eta^t)_{\mathcal{M}_{-1}}\leq 0. 
\end{equation}
We remind the reader that the assumption (K5) is only required when we examine the asymptotic behavior of the solutions (and in that case, (K2) is redundant).

In order to formulate a suitable (abstract) evolution equation for $\eta^t$, we define the linear operator $T_r=-\partial_s$ with the domain
\[
D(T_r)=\{\eta^t\in\mathcal{M}^{(0)}_{-1}:\partial_s\eta^t\in\mathcal{M}^{(0)}_{-1},\ \eta^t(0)=0\}.
\]
It is well-known that $T_r$ is the infinitesimal generator of the right-translation semigroup on $\mathcal{M}_{-1}$; indeed, the following result comes from \cite[Theorem 3.1]{Grasselli&Pata02-2}.

\begin{proposition}  \label{t:generator-T}
The operator $T_r$ with domain $D(T_r)$ is an infinitesimal generator of a strongly continuous semigroup of contractions on $\mathcal{M}_{-1}$, denoted $e^{T_rt}$.
\end{proposition}

As a consequence, we also have (cf., e.g. \cite[Corollary IV.2.2]{Pazy83}).

\begin{corollary}
\label{t:memory-regularity-1} Let $T>0$ and assume $g\in L^{1}(0,T;H^{-1}(\Omega))$.
Then, for every $\eta_{0}\in \mathcal{M}_{-1}$, the Cauchy problem for $\eta^t,$
\begin{equation}
\left\{  \begin{array}{ll}
\partial_{t}\eta^t=T_r\eta^t+g(t), & \text{for}~~t>0, \\
\eta^{0}=\eta_{0}, &
\end{array}  \right.  \label{memory-1}
\end{equation}
has a unique (mild) solution $\eta\in C([0,T];\mathcal{M}_{-1})$ which can be explicitly given as
\begin{equation}
\eta^{t}(s)=\left\{  \begin{array}{ll}
\displaystyle\int_{0}^{s}g(t-y)dy, & \text{for}~~0<s\leq t, \\
\displaystyle\eta_{0}(s-t)+\int_{s}^{t}g(t-y)dy, & \text{for }~~s>t,
\end{array}  \right.  \label{representation-formula-1}
\end{equation}
cf. also \cite[Section 3.2]{CPS06} and \cite[Section 3]{Grasselli&Pata02-2}.
\end{corollary}

\section{Variational formulation and well-posedness}

To begin this section, we state the assumptions on the nonlinear term $F$ and report some important consequences of these assumptions.
These assumptions on $F$ are based on \cite{Frigeri&Grasselli12,Gal&Grasselli14} and can be found in \cite[Section 3]{Gal-NCHEFDBC-16}.

\begin{description}
\item[(N1)] $F\in C^{2}_{loc}(\mathbb{R})$ and there exists $c_F>0$ such that, for all $r\in\mathbb{R},$
\begin{align}
F''(r) \ge -c_F.  \notag
\end{align}

\item[(N2)] There exists $c_F>0$ and $p\in(1,2]$ such that, for all $r\in\mathbb{R},$
\begin{align}
|F'(r)|^p \le c_F (|F(r)| + 1).  \notag
\end{align}

\item[(N3)] There exist $C_1,C_2>0$ such that, for all $r\in\mathbb{R},$
\begin{align}
F(r)\ge C_1|r|^{p/(p-1)} - C_2.  \notag
\end{align}
\end{description}

The last assumption is not needed to obtain the existence of weak solutions, but it will be relied upon later when we seek the existence of strong/regular solutions and uniqueness of these solutions.

\begin{description}
\item[(N4)] There exist $\rho\ge2$ and $C_3>0$ such that, for all $r\in\mathbb{R},$
\begin{align}
|F''(r)|\le C_3(1+|r|^{\rho-2}).
\end{align}
\end{description}

The following remarks are from \cite{Gal-NCHEFDBC-16}.
Assumption (N1) implies that the potential $F$ is a quadratic perturbation of some strictly convex function; i.e., there holds,
\begin{equation}
F(r) = G(r) - \frac{c_F}{2}r^2,  \label{convex}
\end{equation}
with $G\in C^{2}(\mathbb{R})$ strictly convex as $G''\ge 0$ in $\Omega.$
Also with (N1), for each $M\ge0$ there are constants $C_i>0$, $i=3,\dots,6$, (with $C_4$ and $C_5$ depending on $M$ and $F$) such that, for all $r\in\mathbb{R},$
\begin{align}
F(r)-C_3\le C_4(r-M)^2 + F'(r)(r-M),  \label{Fcons-1}
\end{align}
\begin{align}
\frac{1}{2}|F'(s)|(1+|r|) \le F'(r)(r-M) + C_5,  \label{Fcons-2}
\end{align}
(cf. \cite[Equations (4.7) and (4.8)]{CGG11}) and
\begin{align}
|F(r)| - C_6 \le |F'(r)|(1+|r|).  \label{Fcons-3}
\end{align}
The last inequality appears in \cite[page 8]{Gal&Miranville09}.
With the positivity condition (N3), it follows that, for all $r\in\mathbb{R},$
\begin{equation}
|F'(r)| \le c_F(|F(r)| + 1).  \label{Fcons-3.1}
\end{equation}
Assumption (N2) allows for {\em arbitrary} polynomial growth $\bar p=p/(p-1)$ in the potential $F$.
Significantly, the double-well potential $F(r)=(r^2-1)^2$ satisfies (N2) with $p=4/3$ and (N4) with $p=2$.

We are now ready to introduce the variational/weak formulation of Problem {\textbf P}.

\begin{definition}  \label{d:ws}
Let $T>0$ and $\phi_0=(u_0,\eta_0)^{tr}\in\mathcal{H}^M_{\beta,0}=W^{\beta,2}_0(\Omega)\times\mathcal{M}^{(0)}_{-1}$ be such that $F(u_0)\in L^1(\Omega).$
A pair $\phi=(u,\eta)$ satisfying 
\begin{align}
\phi=(u,\eta) & \in L^\infty(0,T;\mathcal{H}^M_{\beta,0}),  \label{wk-reg-1} \\
\partial_t u & \in L^2(0,T;H^{-1}(\Omega)),  \label{wk-reg-2} \\
\partial_t\eta & \in L^2(0,T;H^{-1}_\nu(\mathbb{R}_+;H^{-1}_{(0)}(\Omega))),  \label{wk-reg-2.5} \\
\mu & \in L^2(0,T;W^{-\beta,2}(\Omega)),  \label{wk-reg-3} \\
F'(u)  & \in L^\infty(0,T;L^p(\Omega))  \label{wk-reg-4}
\end{align}
is called a {\sc{weak solution}} to Problem {\textbf P} on $[0,T]$ with initial
data $\phi_0=(u_0,\eta_0)\in\mathcal{H}^M_{\beta,0}$ if the following identities hold almost everywhere in $(0,T)$, and for all $v\in H^1(\Omega),$ $\xi\in W^{\beta,2}_0(\Omega)\cap L^p(\Omega)$ and $\zeta\in\mathcal{M}_1$:
\begin{align}
\langle \partial_tu,v\rangle_{H^{-1}\times H^1} + \int_0^\infty \nu(s)\langle\eta^t(s),v\rangle_{H^{-1}\times H^1} ds & = 0,  \label{var-1} \\
a_{E,\beta}(u,\xi) + \langle F'(u),\xi\rangle_{W^{-\beta,2}\times W^{\beta,2}_0} + \alpha \langle\partial_tu,\xi\rangle_{W^{-\beta,2}\times W^{\beta,2}_0} & = \langle\mu,\xi\rangle_{W^{-\beta,2}\times W^{\beta,2}_0},  \label{var-2} \\
(\partial_t\eta^t,\zeta)_{\mathcal{M}_{-1}} - (T_r\eta^t,\zeta)_{\mathcal{M}_{-1}} & = (\mu,\zeta)_{\mathcal{M}_0}.  \label{var-3}
\end{align}
Also, the initial conditions hold in the $L^2$-sense
\begin{align}
u(0) = u_0 \quad \text{and} \quad \eta^0 = \eta_0.  \label{L2-initial-theta}
\end{align}
Finally, we say that $\phi=(u,\eta)^{tr}$ is a {\sc{global weak solution}} of Problem {\textbf P} if it is a weak solution on $[0,T]$, for any $T>0.$
\end{definition}

\begin{remark}
It is important to note that although $\eta_0$ is defined by \eqref{eta-1} and \eqref{mu-ic}, $\eta_0$ may be taken to be initial data {\em independent} of $u$.
Henceforth we will consider a more general problem with respect to the original one.
\end{remark}

\begin{remark}
Concerning equation \eqref{var-3} and the representation formula (\ref{representation-formula-1}), we have
\begin{equation*}
T_r\eta^{t}(s)=-\partial_{s}\eta^{t}(s)=\left\{  \begin{array}{ll}
-\Delta \mu(t-s) & \text{for}~0<s\leq t, \\
-\partial_{s}\eta_{0}(s-t)-\Delta \mu(t-s) & \text{for}~s>t.
\end{array}  \right.
\end{equation*}
Thus, when given $\eta_{0}\in\mathcal{M}_{-1}^{(0)},$ then $T_r\eta^t\in H_{\nu}^{-1}(\mathbb{R}_{+};H^{-1}(\Omega)),$ for each $t\in(0,T)$, by virtue of \eqref{wk-reg-3}.
Moreover, taking $\zeta=1$ in the variational equation
\begin{align}
(\partial_t\eta^t,\zeta)_{\mathcal{M}_{-1}} - (T_r\eta^t,\zeta)_{\mathcal{M}_{-1}} & = -\int_0^\infty \nu(s)(-\Delta\mu,\zeta)_{H^{-1}\times H^1}ds, \notag
\end{align}
we find, for all $s>t$,
\begin{align}
\frac{\partial}{\partial t}\langle \eta^t(s) \rangle + \frac{\partial}{\partial s}\langle \eta_0(s-t) \rangle + \langle \Delta\mu(t-s) \rangle = \langle \Delta\mu(t-s) \rangle k_0. \notag 
\end{align}
We know that $\eta_0\in\mathcal{M}_{-1}^{(0)}$ and $k_0=1$, hence
\begin{align}
\frac{\partial}{\partial t}\langle \eta^t(s) \rangle = 0, \notag 
\end{align}
and it follows that 
\[
\langle\eta^t(s)\rangle=0\quad\forall t\ge0. \label{pre-mass-con}
\]
\end{remark}

\begin{remark}
In the Cahn-Hilliard model, it is well-known that the average value of $u$ is conserved (cf. e.g. \cite[Section III.4.2]{Temam01}). 
A similar property holds here for our problem.
Indeed, we may choose the test function $v=1$ in \eqref{var-1} which yields
\begin{align}
\frac{\partial}{\partial t}\langle u(t)\rangle + \int_0^\infty\nu(s)\langle \eta^t(s)\rangle ds=0. \notag 
\end{align}
By \eqref{pre-mass-con}, there holds $\langle \eta^t(s)\rangle=0$ for all $t>0$ and for all $s>0$. Hence, we recover {\em{conservation of mass}}
\begin{align}
\langle u(t) \rangle = \langle u_0 \rangle \quad \text{and} \quad \langle \partial_tu(t) \rangle = 0, \quad \forall t\ge0.  \label{mass-con}
\end{align}
\end{remark}

\begin{remark}
Before we continue to the existence statement, it is worthwhile to recall Theorem \ref{t:append-1} (d) in Appendix \ref{s:append-a} for which the following embedding holds
\begin{equation}
D(A_{E,\beta}) \hookrightarrow L^\infty(\Omega), \quad\quad \forall\beta\in(\frac{N}{4},1), \quad \text{for} \ N=1,2,3.  \label{beta-range}
\end{equation}
\end{remark}

\begin{theorem}  \label{t:wk-sols}
Let $T>0$ and $\phi_0=(u_0,\eta_0)^{tr}\in\mathcal{H}^M_{\beta,0}=W^{\beta,2}_0(\Omega)\times\mathcal{M}^{(0)}_{-1}$ for $\beta\in(\frac{N}{4},1)$, $N=1,2,3$, be such that $F(u_0)\in L^1(\Omega).$
Assume $\alpha>0$ and that (K1)-(K4) and (N1)-(N3) hold.
Problem {\textbf P} admits at least one weak solution $\phi=(u,\eta)$ on $(0,T)$ according to Definition \eqref{d:ws} with the additional regularity
\begin{align}
u & \in L^\infty(0,T;L^{p/(p-1)}(\Omega)),  \label{wk-reg-5} \\
\sqrt{\alpha}\partial_tu & \in L^2(\Omega\times(0,T)),  \label{wk-reg-5.5} \\
\eta & \in L^2(0,T;L^2_{-\nu'}(\mathbb{R}_+;H^{-1}_{(0)}(\Omega))),  \label{wk-reg-6} \\
F(u)  & \in L^\infty(0,T;L^1(\Omega)), \quad F'(u)\in L^\infty(0,T;L^1(\Omega)).  \label{wk-reg-7}
\end{align}
for any $T>0.$
Furthermore, setting
\begin{align}
\mathcal{E}(t) := |\|u(t)\||^2_{W^{\beta,2}_0} + 2(F(u(t)),1) + \|\eta^{t}\|^2_{\mathcal{M}_{-1}} + C  \label{wk-1}
\end{align}
for some $C>0$ sufficiently large, the following energy equality holds for every such weak solution,
\begin{align}
\mathcal{E}(t) + 2 \int_0^t \left( \alpha\|\partial_tu(\tau)\|^2 d\tau - \int_0^\infty \nu'(s)\|\eta^\tau(s)\|^2_{H^{-1}}ds \right) d\tau = \mathcal{E}(0).  \label{wk-2}
\end{align}
\end{theorem}

\begin{proof}
The proof proceeds in several steps.
The existence proof begins with a Faedo-Galerkin approximation procedure in which we later pass to the limit.
We first assume that $u_0\in D(A_{E,\beta})$.
(This assumption will be used to show that there is a sequence $\{u_{0n}\}_{n=1}^\infty$ such that $u_{0n}\rightarrow u_{0}$ in $D(A_{E,\beta})$ as well as $L^\infty(\Omega)$ per \eqref{beta-range}, which will be important in light of the fact that $F(u_{0n})$ is of arbitrary polynomial growth per assumptions (N1)-(N3).)
The existence of a weak solution for $u_0\in W^{\beta,2}_0(\Omega)$ with $F(u_0)\in L^1(\Omega)$ will follow from a density argument. 
To establish the equality in the energy identity, we exploit the fact that the potential $F$ is a quadratic perturbation of some strictly convex function.

{\em{Step 1: The Galerkin approximation.}} To begin, we
introduce the family $\{v_j\}_{j\geq1}$ of eigenvectors of the
fractional Laplacian $A_{E,\beta}$ which exist thanks to Theorem \ref{t:append-1} in Appendix \ref{s:append-a}.
Moreover, there is a family $\{w_j\}_{j\geq1}$ consisting of the eigenvectors of
the Neumann-Laplacian $A_N$, and with this, we define the smooth sequence of $\{z_j\}_{j\ge1}\subset D(T_r)\cap W^{1,2}_\nu(\mathbb{R}_+;H^1_{(0)}(\Omega))$ by $z_j=b_jw_j$ such that $\{b_j\}_{j\ge1}\subset C_c^\infty(\mathbb{R}_+)$ is an orthonormal basis for $L^2_\nu(\mathbb{R}_+)$.
Using these we define the following finite-dimensional spaces:
\begin{align} V^n={\rm
span}\{v_1,v_2,\dots,v_n\},\quad W^n={\rm
span}\{w_1,w_2,\dots,w_n\},\quad \mathcal{M}^n={\rm
span}\{z_1,z_2,\dots,z_n\},\quad   \label{eq20}
\end{align}
and set 
\[
V^\infty = \bigcup_{n=1}^\infty V^n, \quad W^\infty = \bigcup_{n=1}^\infty W^n, \quad \mathcal{M}^\infty = \bigcup_{n=1}^\infty \mathcal{M}^n. \quad 
\]
Clearly, $V^\infty$ is a dense subspace of $W^{\beta,2}_0(\Omega)$ and $W^\infty$ is a dense subspace of $H^1(\Omega)$.
In addition, $\mathcal{M}^\infty$ is a dense subspace of $\mathcal{M}^{(0)}_{-1}.$
For $T>0$ fixed, we look for two functions of the form on $(0,T)$,
\begin{equation}
u _{n}(t)=\sum\limits_{k=1}^{n}a_{k}^{(n)}(t)v_{k} \quad \text{and} \quad \eta^t_n(s)=\sum\limits_{k=1}^{n}c_{k}^{(n)}(t) z_{k}, \label{appx-soln}
\end{equation}
where $a_j^{(n)}$ and $c_j^{(n)}$ are assumed to be (at least) $C^2([0,T])$ for each $j=1,2,\dots$ an for each $n=1,2,\dots,$ that solve the following approximating Problem $\mathbf{P}_{n}$:
\begin{align}
(\partial_tu_n,v) + \int_0^\infty \nu(s)(\eta^t_n(s),v)ds & = 0  \label{eq21} \\
a_{E,\beta}(u_n,\xi) + (F'(u_n),\xi) + \alpha(\partial_tu_n,\xi) & = (\mu_n,\xi)  \label{eq22} \\
(\partial_t\eta^t_n,\zeta)_{\mathcal{M}_{-1}} - (T_r\eta^t_n,\zeta)_{\mathcal{M}_{-1}} & = (\mu_n,\zeta)_{\mathcal{M}_0} \label{eq23}  \\ 
u_{n}(0)=u_{0n}, \quad \eta_{n}^{0}=\eta_{0n} \label{eq24}
\end{align}
for every $v\in V^n,$ $\xi\in W^n$ and $\zeta\in \mathcal{M}^n$, and where $u_{0n}$ and $\eta_{0n}$ denote the finite-dimensional projections of $u_0$ and $\eta_0$ onto $V^n$ and $\mathcal{M}^n$, respectively. 
This approximating problem is equivalent to solving a Cauchy problem
for a system of ordinary differential equations (indeed, cf. e.g. \cite[page 131]{CGG11}). 
Hence, the Cauchy-Lipschitz theorem ensures that there exists a
$T_{n}\in(0, \infty]$ such that this approximating system has
a unique maximal solution.

{\em{Step 2: {\em{A priori}} estimates.}}
We now derive some a priori estimates in order to show that
$T_{n}=\infty$ for every $n\geq1$ and that the sequences of
$u_{n}, \eta^t_{n}, \mu_{n}$ are bounded in suitable functional
spaces. By using $v=\mu_n$ as a test function in \eqref{eq21} and
$\xi=\partial_tu_n$ as a test function in \eqref{eq22} we obtain
\begin{align}
& (\partial_tu_n,\mu_n) + \int_0^\infty \nu(s)(\eta^t_n(s),\mu_n)ds = 0  \label{eq5} \\
& (\mu_n,\partial_tu_n) = ((-\Delta)_E^\beta u_n,\partial_tu_n) + (F'(u_n),\partial_tu_n) + \alpha\|\partial_tu_n\|^2,
\label{eq6}
\end{align}
and taking $\zeta=\eta^t_n$ as a test function in \eqref{eq23} yields (for the products in $\mathcal{M}_{-1}$, this is multiplication by $(-\Delta)^{-1}\eta^t_{n}$ in
$\mathcal{M}_0$)
\begin{align}
\int_0^\infty\nu(s)\left(\int_\Omega\partial_t\eta_n^t(x,s)(-\Delta)^{-1}\eta_n^t(x,s)dx\right)ds
& +
 \int_0^\infty\nu(s)\left(\int_\Omega\partial_s\eta_n^t(x,s)(-\Delta)^{-1}\eta_n^t(x,s)dx\right)ds  \notag \\
& =
\int_0^\infty\nu(s)\left(\int_\Omega(-\Delta)\mu_n(x,t)(-\Delta)^{-1}\eta_n^t(x,s)dx\right)ds,  \notag
\end{align}
which is, after an integration by parts, 
\begin{align}
(\partial_t\eta^t_n,\eta^t_n)_{\mathcal{M}_{-1}}+(\partial_s\eta^t_n,\eta^t_n)_{\mathcal{M}_{-1}}=(\mu_n,\eta^t_n)_{\mathcal{M}_0}.
\label{eq7}
\end{align}
Then combining the results produces the differential identity, which holds for almost all $t\in(0,T),$
\begin{align}
\frac{1}{2}\frac{d}{dt} \left\{ |\|u_n\||^2_{W^{\beta,2}_0} +
2(F(u_n),1) + \|\eta^t\|^2_{\mathcal{M}_{-1}} \right\} +
\alpha\|\partial_tu_n\|^2 -
(T_r\eta^t_n,\eta^t_n)_{\mathcal{M}_{-1}} = 0.  \label{eq8}
\end{align}
For all $t\in(0,T_n),$ set
\begin{equation}
\mathcal{E}_{n}(t) := |\|u_n(t)\||^2_{W^{\beta,2}_0} + 2(F(u_n(t)),1) + \|\eta_{n}^{t}\|^2_{\mathcal{M}_{-1}} + C  \label{energy}
\end{equation}
where, in light of (N3), the constant $C>0$ may be taken sufficiently large (i.e. $C>C_2|\Omega|$) in order to ensure that $\mathcal{E}_{n}(t)$ is nonnegative for all $t\in(0,T_n)$. We have
\begin{align}
\frac{d}{dt} \mathcal{E}_{n} + 2\alpha\|\partial_tu_n\|^2 - 2\int_0^\infty \nu'(s)\|\eta^t_n(s)\|^2_{H^{-1}}ds = 0 \label{eq9}
\end{align}
for almost all $t\in (0,T_{n})$. Hence, integrating the equation above with respect to time in $(0,t)$, we are led to the following integral
equality (which does hold for the approximate solutions)
\begin{align}
\mathcal{E}_{n}(t)+2\int_0^t\left(\alpha\|\partial_tu_n(\tau)\|^2-\int_0^\infty\nu'(s)\|\eta^\tau_n(s)\|^2_{H^{-1}}ds \right) d\tau=
\mathcal{E}_{n}(0).\label{eq25}
\end{align}
Furthermore, from \eqref{energy} and assumption (N3), we find the lower bound
\begin{align}
|\|u_n(t)\||^2_{W^{\beta,2}_0} + 2C_1\|u_n(t)\|^{p/(p-1)}_{L^{p/(p-1)}} + \|\eta_{n}^{t}\|^2_{\mathcal{M}_{-1}} \leq \mathcal{E}_{n}(t).  \label{eq10}
\end{align}
Using the fact that $F(u_0)\in L^1(\Omega)$, we also obtain the upper bound
\begin{align}
\mathcal{E}_{n}(t) \le \mathcal{E}_{n}(0) & \leq |\|u_n(0)\||^2_{W^{\beta,2}_0} + (F(u_n(0)),1) + \|\eta_{n}^{0}\|^2_{\mathcal{M}_{-1}}  \notag \\
& \le Q(\|\phi_n(0)\|_{\mathcal{H}^M_{\beta,0}})+C.  \label{eq10.5}
\end{align}
In particular, the uniform bound derived from \eqref{eq25}-\eqref{eq10.5} implies that the local solution to
Problem $\mathbf{P_{n}}$ can be extended up to time $T$, that is
$T_{n}=T$, for every $n$.
Moreover, from \eqref{eq25}-\eqref{eq10} we deduce the
following bounds for the approximate solution
\begin{align}
\|u_n\|_{L^\infty(0,T;W^{\beta,2}_0)} & \le C  \label{eq26} \\
\|\eta_n\|_{L^\infty(0,T;\mathcal{M}_{-1})} & \le C  \label{eq27} \\
\|F(u_n)\|_{L^\infty(0,T;L^1)} & \le C  \label{eq28} \\
\sqrt{\alpha}\|\partial_tu_n\|_{L^2(\Omega\times(0,T))} & \le C  \label{eq29} \\
\|\eta_n\|_{L^2(0,T;L^2_{-\nu'}(\mathbb{R}_+;H^{-1}))} & \le C  \label{eq30} \\
\|u_n\|_{L^\infty(0,T;L^{p/(p-1)})} &
\le C.  \label{eq30.1}
\end{align}
Obviously, \eqref{Fcons-3.1} and \eqref{eq28} immediately show us
\begin{align}
\|F'(u_n)\|_{L^\infty(0,T;L^1)} & \le C.  \label{eq30.2}
\end{align}
Next, since $\langle A_N^{-1}\partial_tu_n\rangle=0$ (recall \eqref{mass-con}$_2$), we may (and do) take $v=A_N^{-1}\partial_tu_n$ in \eqref{eq21} which leads us to the estimate,
\begin{align}
\|A_N^{-\frac{1}{2}}\partial_{t}u_{n}\|^{2} \leq \int_0^\infty
\nu(s)\|A_N^{-\frac{1}{2}}\eta_{n}^{t}(s)\|
\|A_N^{-\frac{1}{2}}\partial_tu_n(t)\|ds,
\end{align}
that is,
\begin{align}
\|\partial_{t}u_{n}\|_{H^{-1}}^{2}\leq\int_0^\infty
\nu(s)\|\eta_{n}^{t}(s)\|_{H^{-1}} \|\partial_tu_n\|_{H^{-1}}ds.
\end{align}
Using Young's inequality and assumption (K3), we can write
\begin{align}
\|\partial_tu_n\|_{H^{-1}}\leq \|\eta_{n}^{t}\|_{\mathcal{M}_{-1}}.
\label{eq00}
\end{align}
Thus, \eqref{eq27} and \eqref{eq00} yield
\begin{align}
\|\partial_tu_n\|_{L^\infty(0,T;H^{-1})} \le C.  \label{eq31}
\end{align}
Need to bound $F'(u_n)$, then $\mu_n$.
In light of \eqref{eq22}, we apply \eqref{eq30.2}, \eqref{eq31}, and the fact that operator $A_{E,\beta}$ is bounded from $W^{\beta,2}_0(\Omega)$ into $W^{-\beta,2}(\Omega)$ (in particular, $\|A_{E,\beta}u_n\|_{L^2(0,T;W^{-\beta,2}(\Omega))}\le C)$, to obtain the following uniform bounds for $\mu_n$
\begin{align}
|\langle \mu_n \rangle| \le C,  \label{eq31.2}
\end{align}
and
\begin{align}
\|\mu_n\|_{L^2(0,T;W^{-\beta,2}(\Omega))} \le C.  \label{eq32}
\end{align}
This completes Step 2.

\emph{Step 3: Passage to the limit.}
On account of the above uniform inequalities, we can argue that there are functions $u$, $\eta,$ $\mu$, such that, up to subsequences,
\begin{eqnarray}
u_{n}\rightharpoonup u & \quad \text{weakly-* in} \quad & L^{\infty}(0,T;W^{\beta,2}_0(\Omega)),  \label{con1} \\
u_{n}\rightharpoonup u & \quad \text{weakly-* in} \quad & L^{\infty}(0,T;L^{p/(p-1)}(\Omega)),  \label{con1.1} \\
\partial_{t}u_{n}\rightharpoonup \partial_{t}u & \quad \text{weakly-* in} \quad & L^{\infty}(0,T;H^{-1}(\Omega)),  \label{con2} \\
\sqrt{\alpha}\partial_{t}u_{n}\rightharpoonup \sqrt{\alpha}\partial_{t}u & \quad \text{weakly in} \quad & L^{2}(\Omega\times(0,T)),  \label{con5} \\
\eta_{n}\rightharpoonup \eta & \quad \text{weakly-* in} \quad & L^{\infty}(0,T;\mathcal{M}_{-1}),  \label{con3} \\
\eta_{n}\rightharpoonup \eta & \quad \text{weakly in} \quad & L^{2}(0,T;L^{2}_{-\nu'}(\mathbb{R}_{+};H^{-1}(\Omega))),  \label{con4} \\
\partial_t\eta_n\rightharpoonup \partial_t\eta & \quad \text{weakly in} \quad & L^2(0,T;H^{-1}_\nu(\mathbb{R}_+;H^{-1}(\Omega))),  \label{con5.1} \\
\mu_n\rightharpoonup \mu & \quad \text{weakly in} \quad & L^{2}(0,T;W^{-\beta,2}(\Omega)).  \label{con6}
\end{eqnarray}
(Note that \eqref{con5.1} is due to \eqref{eq23} and the definition of the the operator $T_r$.) 
Using the above convergences \eqref{con1} and \eqref{con2}, as well as the fact that the injection $W^{\beta,2}_0(\Omega)\hookrightarrow L^2(\Omega)$ is compact for any $\beta\in(0,1)$, we draw upon the conclusion of the Aubin-Lions Lemma (cf. Lemma \ref{t:Lions} in Appendix \ref{s:append-a}) to deduce the following embedding is compact
\begin{align}
W:=\{ \chi\in L^2(0,T;W^{\beta,2}_0(\Omega)) : \partial_t\chi\in L^2(0,T;H^{-1}(\Omega)) \} \hookrightarrow L^2(\Omega\times(0,T)).  \label{compact-u}
\end{align}
Hence,
\begin{align}
u_{n}\rightarrow u \quad \text{strongly in} \quad L^{2}(\Omega\times(0,T)), \label{st-con-u}
\end{align}
and we deduce that $u_{n}$ converges to $u$, almost everywhere in
$\Omega\times (0,T)$. Using assumption (N1) with \eqref{st-con-u}, we deduce
\begin{align}
F'(u_{n})\rightarrow F'(u) \quad \text{strongly in} \quad L^{2}(0,T;L^1(\Omega)).  \label{con8}
\end{align}
Thus, we now have all the sufficient convergence resuts to pass to the limit in equations \eqref{eq21} and
\eqref{eq22} in order to recover \eqref{p-1} and \eqref{p-2},
respectively.
It remains to recover equation \eqref{eq23} after we pass to the limit. 
An integration by parts on the first term in \eqref{eq23} and then an application of \eqref{con3} yields, for
any $\zeta \in C_{0}^{\infty}((0,T);C_{0}^{\infty}((0,T);H^{1}(\Omega)))$
\begin{align}
\int_0^{T}(\partial_{t}\eta_{n}^{\tau},\zeta)_{\mathcal{M}_{-1}}d\tau=
-\int_0^{T}(\eta_{n}^{\tau},\partial_{t}\zeta)_{\mathcal{M}_{-1}}d\tau\rightarrow
-\int_0^{T}(\eta^{\tau},\partial_{t}\zeta)_{\mathcal{M}_{-1}}d\tau.
\label{con9}
\end{align}
With this we have
\begin{align}
\partial_{t}\eta_{n}^{t}\rightharpoonup \partial_{t}\eta^{t} \quad \text{weakly in} \quad L^2(0,T;H^{-1}_{\nu}(\mathbb{R}_+;H^{-1}(\Omega))) \label{con10}
\end{align}
and that $\eta^{t}\in
L^\infty(0,T;H^{-1}_{\nu}(\mathbb{R}_+;H^{-1}(\Omega)))$. Furthermore,
with the help of \eqref{con4}, we have
\begin{align}
-\int_0^{T}(T_r\eta_{n}^{\tau},\zeta)_{\mathcal{M}_{-1}}d\tau=-\int_0^{T}\nu'(s)(\eta_{n}^{\tau},\zeta)_{H^{-1}}d\tau\rightarrow
-\int_0^{T}\nu'(s)(\eta^{\tau},\zeta)_{H^{-1}}d\tau.  \label{con11}
\end{align}
By using a density argument (cf. \cite{Grasselli&Pata02-2}) and the following distributional equality 
\begin{align}
-\int_0^{T}(\eta_{n}^{\tau},\partial_{t}\zeta)_{\mathcal{M}_{-1}}d\tau-\int_0^{T}\nu'(s)(\eta^{\tau},\zeta)_{H^{-1}(\Omega)}d\tau
=\int_0^{T}(\partial_{t}\eta^{\tau}-T_r\eta^{\tau},\zeta)_{\mathcal{M}_{-1}}d\tau,
\end{align}
we also get \eqref{eq23} on account of \eqref{con3} and \eqref{con6}.
This completes Step 3 of the proof.

{\em Step 4: Energy equality.}
To begin, let $u_0\in D(A_{E,\beta})$, $\eta_0\in \mathcal{M}^{(0)}_{-1}$ and let $\phi=(u,\eta)^{tr}$ be the corresponding weak solution.
Recall from \eqref{st-con-u}, we have, for almost all $t\in(0,T),$
\begin{align}
u_n(t)\rightarrow u(t) \quad \text{strongly in $L^2(\Omega)$ and a.e. in $\Omega.$}  \label{wk-50}
\end{align}
Since $F$ is measurable (see (N1)), Fatou's lemma implies
\begin{equation}
\int_\Omega F(u(t)) dx \le \liminf_{n\rightarrow+\infty}\int_\Omega F(u_n(t)) dx.  \label{wk-51}
\end{equation}
Passing to the limit in \eqref{eq25}, and while keeping in mind \eqref{con1}, \eqref{con3}, \eqref{con5}, \eqref{con5.1}, \eqref{con6} and \eqref{con8}, as well as the weak lower-semicontinuity of the norm, we arrive at the integral inequality which holds for any weak solution
\begin{align}
\mathcal{E}(t) + 2 \int_0^t \left( \alpha\|\partial_tu(\tau)\|^2 d\tau - \int_0^\infty \nu'(s)\|\eta^\tau(s)\|^2_{H^{-1}}ds \right) d\tau \le \mathcal{E}(0).  \notag 
\end{align}

We argue as in the proof of \cite[Corollary 2]{CFG12} to establish the energy equality.
Indeed, take $\xi=\mu$ in \eqref{var-1}. 
By \eqref{p-2}, we need to treat the dual pairing $\langle F'(u),\partial_tu \rangle_{W^{-\beta,2}\times W^{\beta,2}_0}$.
It is here where we employ \eqref{convex} where $F'(u)=G'(u)-c_Fu$ and $G'\in C^1(\mathbb{R})$ is monotone increasing.
Define the functional $\mathcal{G}:L^2(\Omega)\rightarrow\mathbb{R}$ by
\[
\mathcal{G}(\phi):=\left\{ \begin{array}{ll} \displaystyle\int_\Omega G(u)dx & \text{if}\ G(u)\in L^1(\Omega), \\ +\infty & \text{otherwise}. \end{array} \right.
\]
Now by \cite[Proposition 2.8, Chapter II]{Barbu76}, it follows that $\mathcal{G}$ is convex, lower-semicontinuous on $L^2(\Omega)$, and $\chi\in\partial\mathcal{G}(u)$ if and only if $\chi=G'(u)$ almost everywhere in $\Omega$.
Since we have \eqref{wk-reg-2}, we apply \cite[Proposition 4.2]{CKRS07} to find that there holds, for almost all $t\in(0,T),$ 
\begin{align}
\langle \partial_tu,F'(u) \rangle_{W^{-\beta,2}\times W^{\beta,2}_0} & = \langle \partial_tu,G'(u) \rangle_{W^{-\beta,2}\times W^{\beta,2}_0} - c_F\langle \partial_tu,u \rangle_{W^{-\beta,2}\times W^{\beta,2}_0}  \notag \\ 
& = \frac{d}{dt} \left\{ \mathcal{G}(u) - \frac{c_F}{2}\|u\|^2 \right\}  \notag \\ 
& = \frac{d}{dt} \int_\Omega F(u) dx.  \notag 
\end{align}

Similar to Step 2 above, take $v=\mu$, $\xi=\partial_tu$ and $\zeta=\eta^t$ (now without the index $n$) in \eqref{var-1}-\eqref{var-3}, respectively. 
Using the above result on the dual product with $F'(u)$ and \eqref{wk-reg-2}, we are led to the differential identity \eqref{eq9} with $E$, $u$ and $\eta$ in place of $\mathcal{E}_n$, $u_n$ and $\eta_n$, respectively.
Integrating the resulting differential identity on $(0,t)$ produces \eqref{wk-2} as claimed.
This completes Step 4. 

{\em Step 5: $(u,\eta)$ weak solution to Problem {\textbf P}.}
Now let us take $\phi_0=(u_0,\eta_0)^{tr}\in \mathcal{H}^M_{\beta,0}$ where $F(u_0)\in L^1(\Omega).$
Proceeding exactly as in \cite[page 440]{CFG12} the bounds \eqref{eq26}-\eqref{eq30.2} and \eqref{eq31}-\eqref{eq32} hold.
Moreover, with the aid of the Aubin-Lions compact embedding (again see Lemma \ref{t:Lions} in Appendix \ref{s:append-a} below) we deduce the existence of functions $u$, $\eta$ and $\mu$ that satisfy \eqref{wk-reg-1}, \eqref{wk-reg-3}, \eqref{wk-reg-5} and \eqref{wk-reg-6}.
Thus, passing to the limit in the variational formulation for $\phi_k=(u_k,\eta_k)^{tr}$, we find $\phi=(u,\eta)^{tr}$ is a solution corresponding to the initial data $\phi_0=(u_0,\eta_0)^{tr}\in\mathcal{H}_{\beta,0}^M$ for which $F(u_0)\in L^1(\Omega).$
This finishes the proof of the theorem.
\end{proof}

Before we continue, we make some important remarks.

\begin{remark}
The continuity property
\begin{align}
u\in C([0,T];W_0^{\beta-\iota,2}(\Omega)), \notag
\end{align}
for any $\iota>0$, sufficiently small follows from the conditions in Definition \ref{d:ws} after an application of the Aubin-Lions Lemma (cf. Lemma \ref{t:Lions} in Appendix \ref{s:append-a}).
In addition, the property
\begin{align}
\eta\in C([0,T];\mathcal{M}_{-1}^{(0)}) \notag
\end{align}
follows from the density argument in \cite{Grasselli&Pata02-2}. 
Thus, we deduce the continuity properties
\begin{align}
\phi=(u,\eta) \in C([0,T];\mathcal{H}^{M}_{\beta,0}). \notag 
\end{align}
\end{remark}

\begin{remark}  \label{r:gradient}
From \eqref{wk-2} we see that if there is a $t^*>0$ in which
\[
\mathcal{E}(t^*)=\mathcal{E}(0),
\]
then, for all $t\in(0,t^*),$
\begin{equation}  \label{no-grad}
\int_0^t \left( \alpha\|\partial_tu(\tau)\|^2 + \|\eta^\tau\|^2_{L^2_{-\nu'}(\mathbb{R}_+;H^{-1})} \right) d\tau = 0.
\end{equation}
We deduce $\partial_tu(t)=0$ for all $t\in(0,t^*)$.
Additionally, since $u(t)=u_0$ for all $t\in(0,t^*),$ equation \eqref{p-2} shows
\[
\mu(t)=A_{E,\beta}u_0+F'(u_0) \quad \forall t\in(0,t^*),
\]
i.e., $\mu(t)=\mu^*$ is also stationary.
Thus, by the definition of $\eta^t$ given in \eqref{eta-1}, we find here that, for each $t\in(0,t^*)$
\[
\eta^t(s)=s A_N\mu^* \quad \forall s\ge0.
\]
Therefore $\phi=(u,\eta)^{tr}$ is a fixed point of the trajectory $\phi(t)=\mathcal{S}(t)\phi_0$, where $\mathcal{S}$ is the solution operator defined below in Corollary \ref{t:Lip-semif}.
\end{remark}

The following result (cf. \cite[Theorem 3.4]{CGG11}) concerns the existence of strong/regular solutions which will be utilized in the proof of the continuous dependence estimate.
Note that we will now employ the added assumption on the nonlinear term.

\begin{theorem} \label{t:st-solns}
Let $T>0$ and $\phi_0=(u_0,\eta_0)^{tr}\in \mathcal{H}^M_{\beta+1,\beta+1}:=W^{\beta+1,2}_0(\Omega)\times L^2_\nu(\mathbb{R}_+;W^{\beta,2}_0(\Omega))$ be such that $F(u_0)\in L^1(\Omega)$ and $\eta_0\in D(T_r)$.
Assume $\alpha>0$ and that (K1)-(K4) and (N1)-(N3) hold.
Additionally, assume that (N4) holds.
Problem {\textbf P} admits at least one weak solution $\phi=(u,\eta)$ on $(0,T)$ according to Definition \eqref{d:ws} with the additional regularity, for any $T>0$,
\begin{align}
\phi=(u,\eta) & \in L^\infty(0,T;\mathcal{H}^M_{\beta+1,\beta+1})\cap W^{1,\infty}(0,T;\mathcal{H}_{\beta,0}^M),  \label{st-reg-1} \\
\sqrt{\alpha}\partial_tu & \in L^2(0,T;H^1(\Omega)) \\
\partial_{tt}u & \in L^\infty(0,T;H^{-1}(\Omega)),  \label{st-reg-2} \\
\sqrt{\alpha}\partial_{tt}u & \in L^2(\Omega\times(0,T)), \label{st-reg-3} \\
\mu & \in L^\infty(0,T;H^1(\Omega)), \label{st-reg-3.5} \\
\eta & \in L^\infty(0,T;D(T_r)).  \label{st-reg-4}
\end{align}
\end{theorem}

\begin{proof}
The proof relies on the Galerkin approximation scheme developed in the proof of Theorem \ref{t:wk-sols}.
We will seek $\phi_n=(u_n,\eta_n)$ of the form \eqref{appx-soln} satisfying Problem $\mathbf{P}_{n}$:
\begin{align}
(\partial_{tt}u_n,v) + \int_0^\infty \nu(s)(\partial_t\eta^t_n(s),v)ds & = 0  \label{app1} \\
a_{E,\beta}(\partial_tu_n,\xi) + (F''(u_n)\partial_tu,\xi) + \alpha(\partial_{tt}u_n,\xi) & = (\partial_t\mu_n,\xi)  \label{app2} \\
(\partial_{tt}\eta^t_n,\zeta)_{\mathcal{M}_{-1}} -
(T_r\partial_t\eta^t_n,\zeta)_{\mathcal{M}_{-1}} & =
(\partial_t\mu_n,\zeta)_{\mathcal{M}_0} \label{app3}
\end{align}
for every $t\in(0,T),$ $v\in V^n$, $\xi\in W^n$ and $\zeta\in\mathcal{M}^n,$ and which satisfy the initial conditions
\begin{align}
u_{n}(0)=\tilde u_{0n} \quad \text{and} \quad \eta_{n}^{0}=\tilde\eta_{0n}, \label{app4}
\end{align}
where we have set 
\begin{equation}
\tilde u_{0n}:= -\int_0^\infty\nu(s)\eta_{0n}(s)ds, \label{app5}
\end{equation}
and
\begin{equation}
\tilde\eta_{0n}:=T_r\eta_{0n}+A_N\mu_{0n}, \label{app6}
\end{equation}
and also
\begin{equation}
\mu_{0n}=-\alpha\int_0^\infty\nu(s)\eta_{0n}(s)ds+A_{E,\beta}u_{0n}+F'(u_{0n}). \label{app6.5}
\end{equation}
It is important to note that when $\phi_0=(u_0,\eta_0)$ satisfies the assumptions of Theorem \ref{t:st-solns}, then it is guaranteed that $(\tilde u_0,\tilde\eta_0)\in\mathcal{H}^M_{1,0}$.
Indeed, relying on the fact that $\|(u_{0n},\eta_{0n})\|_{\mathcal{H}^M_{\beta,0}}\le\|(u_{0},\eta_{0})\|_{\mathcal{H}^M_{\beta,0}}$, we easily obtain the estimate $\|(\partial_tu_n(0),\partial_t\eta^0_n)\|_{\mathcal{H}^M_{\beta,0}}\le Q(\|(u_0,\eta_0)\|_{\mathcal{H}^M_{\beta+1,\beta+1}})$.
Now, for any fixed $n\in\mathbb{N}$, we find a unique local maximal solution $\phi_n=(u_n,\eta_n)\in C^2([0,T_n];\mathcal{H}^M_{\beta+1,2}).$
Next we integrate \eqref{app1} and \eqref{app2} with respect to time on $(0,t)$ and argue as in the proof of Theorem \ref{t:wk-sols} to find the uniform bounds \eqref{eq26}-\eqref{eq30.2}, \eqref{eq31} and \eqref{eq32}. 
In order to obtain the required higher-order estimates, let us begin by labeling 
\[
\tilde u(t)=\partial_tu(t), \quad \tilde\eta^t=\partial_t\eta^t, \quad \tilde\mu(t)=\partial_t\mu(t),
\]
where we are also dropping the index $n$ for the sake of simplicity. 
Then $(\tilde u,\tilde\eta)$ solves the system 
\begin{align}
\langle \partial_t\tilde u,v\rangle_{H^{-1}\times H^1} + \int_0^\infty \nu(s)\langle\tilde\eta^t(s),v\rangle_{H^{-1}\times H^1} ds & = 0,  \label{app7} \\
a_{E,\beta}(\tilde u,\xi) + (F''(u)\tilde u,\xi) + \alpha(\partial_t\tilde u,\xi) & = \langle\mu,\xi\rangle_{W^{-\beta,2}\times W^{\beta,2}_0},  \label{app8} \\
(\partial_t\tilde\eta^t,\zeta)_{\mathcal{M}_{-1}} - (T_r\tilde\eta^t,\zeta)_{\mathcal{M}_{-1}} & = (\tilde\mu,\zeta)_{\mathcal{M}_0},  \label{app9}
\end{align}
for all $v\in H^1(\Omega)$, $\xi\in W^{\beta,2}_0(\Omega)$ and $\zeta\in\mathcal{M}_1$, with the initial conditions
\begin{align}
\tilde u(0)=\tilde u_0 \quad \text{and} \quad \tilde\eta^0=\tilde\eta_0. \notag
\end{align}

Let us now take $v=\tilde\mu$, $\xi=\partial_t\tilde u$ and $\zeta=\tilde\eta^t$ in \eqref{app7}-\eqref{app9}, respectively. 
Summing the resulting identities together, we obtain, for all $t\in(0,T)$,
\begin{align}
\frac{1}{2}\frac{d}{dt}\left\{ \|\tilde u\|^2_{W^{\beta,2}_0} + \|\tilde\eta^t\|^2_{\mathcal{M}_{-1}} \right\} - \int_0^\infty\nu'(s)\|\tilde\eta^t(s)\|^2_{H^{-1}}ds+\alpha\|\partial_t\tilde u\|^2=-(F''(u)\tilde u,\partial_t\tilde u).  \notag
\end{align}
Here we apply (K5) as well as (N4) with \eqref{eq30.1} and the embedding $W^{\beta,2}_0(\Omega)\hookrightarrow L^2(\Omega)$ to find
\begin{align}
\frac{1}{2}\frac{d}{dt}\left\{ \|\tilde u\|^2_{W^{\beta,2}_0} + \|\tilde\eta^t\|^2_{\mathcal{M}_{-1}} \right\} + \lambda\|\tilde\eta^t\|^2_{\mathcal{M}_{-1}}+\alpha\|\partial_t\tilde u\|^2 & \le C_\alpha\|\tilde u\|^2+\frac{\alpha}{2}\|\partial_t\tilde u\| \notag \\ 
& \le C_\alpha\|\tilde u\|^2_{W^{\beta,2}_0}+\frac{\alpha}{2}\|\partial_t\tilde u\|, \label{app10}
\end{align}
where $C_\alpha\sim\alpha^{-1}$ is a positive constant.
Integrating \eqref{app10} over $(0,t)$ produces
\begin{align}
& \|\tilde u(t)\|^2_{W^{\beta,2}_0} + \|\tilde\eta^t\|^2_{\mathcal{M}_{-1}} + \int_0^t \left( 2\lambda\|\tilde\eta^\tau\|^2_{\mathcal{M}_{-1}} + \alpha\|\partial_t\tilde u(\tau)\|^2 \right) d\tau \notag \\
& \le \|\tilde u(0)\|^2_{W^{\beta,2}_0} + \|\tilde\eta^0\|^2_{\mathcal{M}_{-1}} + C_\alpha\int_0^t \|\tilde u(\tau)\|^2_{W^{\beta,2}_0} d\tau, \label{app11}
\end{align}
and an application of Gr\"onwall's (integral) inequality shows, for all $t\ge0$,
\begin{align}
\|(\tilde u(t),\tilde\eta^t)\|_{\mathcal{H}^M_{\beta,0}} \le Q(\|(\tilde u_0,\tilde\eta_0)\|_{\mathcal{H}^M_{\beta,0}}) \label{app12}
\end{align}
and
\begin{align}
\sqrt{\alpha}\|\partial_t\tilde u(t)\|_{L^2(\Omega\times(0,T))} \le Q(\|(\tilde u_0,\tilde\eta_0)\|_{\mathcal{H}^M_{\beta,0}}). \label{app13}
\end{align}
Through \eqref{app5}-\eqref{app6.5} we find $\|(\tilde u_0,\tilde \eta_0)\|_{\mathcal{H}^M_{\beta,0}}$ depends on 
\[
\int_0^\infty\nu(s)\|\eta_0(s)\|^2_{W^{\beta,2}_0}ds, \quad \|A_N\mu_0\|_{\mathcal{M}_{-1}} \quad \text{and} \quad \|T_r\eta_0\|_{\mathcal{M}_{-1}},
\]
hence the assumption on the initial data is justified.

Furthermore, we now consider \eqref{eq23} and take $\zeta=A_N\bar\mu(t)$ where $\bar\mu=\mu-\langle\mu\rangle,$ so that, with \eqref{eq27}, \eqref{eq30} and \eqref{app12}, we obtain, for all $t\ge0$ and for every $\ep>0,$ 
\begin{align}
\|\nabla\mu\|^2 & = (\partial_t\eta^t,\mu)_{\mathcal{M}_{0}} - (T_r\eta^t,\mu)_{\mathcal{M}_{0}} \\
& = \int_0^\infty\nu(s)(\partial_t\eta^t(s),\mu(t))ds - \int_0^\infty\nu'(s)(\eta^t(s),\mu(t))ds \\
& \le C_\ep\left( \|\partial_t\eta^t\|^2_{\mathcal{M}_{-1}} - \int_0^\infty\nu'(s)\|\eta^t(s)\|^2_{H^{-1}}ds \right) + \ep\|\nabla\mu\|^2 \\ 
& \le C_\ep\left( 1-\int_0^\infty\nu'(s)\|\eta^t(s)\|^2_{H^{-1}}ds \right) + \ep\|\nabla\mu\|^2 \\ 
& \le C_\ep+\ep\|\nabla\mu\|^2 \label{app14}
\end{align}
where $C_\ep\sim\ep^{-1}.$
Together \eqref{eq31.2} and \eqref{app14} show us, for all $t\ge0$,
\begin{align}
\|\mu(t)\|_{H^1}\le C. \label{app15}
\end{align}

At this point we can reason as is in the proof of Theorem \ref{t:wk-sols} to find that there is a solution $\phi=(u,\eta)\in W^{1,\infty}(0,T;\mathcal{H}^M_{\beta,0})$ to Problem {\bf P} satisfying \eqref{st-reg-2} and \eqref{st-reg-3}.
Additionally, thanks to \eqref{app15}, the condition \eqref{st-reg-3.5} holds.
It remains to show that 
\[
\phi=(u,\eta)\in L^\infty\left(0,T;\left[W^{\beta+1,2}_0(\Omega)\times L^2_\nu(\mathbb{R}_+;W^{\beta,2}_0(\Omega))\right]\right).
\]
First, in light of \eqref{app12} we multiply \eqref{p-1} by $A_{E,\beta}\eta^t$ in $L^2(\Omega)$ which yields
\begin{align}
\|\eta^t\|^2_{L^2_\nu(\mathbb{R}_+;W^{\beta,2}_0(\Omega))} = -\int_0^\infty \nu(s)(A_{E,\beta}^\frac{1}{2}\partial_tu(t),A_{E,\beta}^\frac{1}{2}\eta^t(s))ds. \notag
\end{align}
Hence, $\eta\in L^\infty(0,T;L^2_\nu(\mathbb{R}_+;W^{\beta,2}_0(\Omega)))$.
Next we consider the identity \eqref{var-2} whereby we may now rely on the regularity properties of $\partial_tu$ and $\mu$.
We take $\xi=A_{N}\partial_tu$ to produce
\begin{align}
\frac{1}{2}\frac{d}{dt}|\|u\||^2_{W^{\beta+1,2}_0}+\langle F''(u)\nabla u,\nabla u\rangle+\alpha\|\partial_tu\|^2_{H^1}=\langle \nabla\mu,\nabla u \rangle. \notag
\end{align}
After applying (N1) and integrating the resulting differential inequality with respect to $t$ over $(0,t)$, there holds, for all $t\ge0,$
\begin{align}
|\|u(t)\||^2_{W^{\beta+1,2}_0} + 2\int_0^\infty\alpha\|\partial_tu(\tau)\|^2_{H^1}d\tau \le |\|u(0)\||^2_{W^{\beta+1,2}_0} + Q(\|(u_0,\eta_0)\|_{\mathcal{H}^M_{\beta,0}}). \notag
\end{align}
We now deduce
\[
u\in L^\infty(0,T;W^{\beta+1,2}_0(\Omega)) \quad \text{and} \quad \sqrt{\alpha}\partial_tu\in L^2(0,T;H^1(\Omega)).
\]
This completes the proof.
\end{proof}

The following proposition provides continuous dependence and uniqueness for the solutions constructed above. 

\begin{proposition}  \label{t:cont-dep}
Let the assumptions of Theorem \ref{t:wk-sols} hold.
Additionally, assume (N4) holds. 
Let $T>0$ and let $\phi_i=(u_{i},\eta_{i})^{tr}$, $i=1,2$, be two solutions to Problem \textbf{P} on $(0,T)$ corresponding to the initial data $\phi_{0i}=(u_{0i},\eta_{0i})^{tr}\in\mathcal{H}^M_{\beta,0}=W^{\beta,2}_0(\Omega)\times\mathcal{M}^{(0)}_{-1}$, such that $F(u_{0i})\in L^1(\Omega)$, $i=1,2$. 
Then, for each $\alpha>0$ there is a positive constant $C_\alpha\sim\alpha^{-1}$ such that the following estimate holds, for any $t\in(0,T)$,
\begin{align}
\|\phi_1(t)-\phi_2(t)\|^2_{\mathcal{H}^M_{\beta,0}} & + \int_0^t \left( \alpha\|\partial_tu_1(\tau)-\partial_tu_2(\tau)\|^2 + \|\eta_1^\tau-\eta_2^\tau\|^2_{L^2_{-\nu'}(\mathbb{R}_+;H^{-1})} \right) d\tau  \notag \\ 
& \leq e^{C_\alpha t}\|\phi_{01}-\phi_{02}\|^2_{\mathcal{H}^M_{\beta,0}}.  \label{cp0}
\end{align}
\end{proposition}

\begin{proof} 
To begin, we assume $(u_{0i},\eta_{0i})$, $i=1,2$, satisfy the assumptions of Theorem \ref{t:st-solns} (recall, above we are assuming (N4) holds), and we will work with the more regular solutions to obtain \eqref{cp0}.
For all $t\in[0,T]$, we then set
\begin{align*}
\phi(t):=\phi_1(t)-\phi_2(t), \quad u(t):=u_{1}(t)-u_{2}(t), \quad \eta^t:=\eta^t_{1}-\eta^t_{2} \quad \text{and} \quad \mu:=\mu_1-\mu_2
\end{align*}
where $\phi_i(t)=(u_{i}(t),\eta^t_{i})$ is a solution corresponding to
$(u_{0i},\eta_{0i})$, $i=1,2.$ 
Then, formally, $\phi=(u,\eta)$ solves the equations for all $v\in H^1(\Omega),$ $\xi\in W^{\beta,2}_0(\Omega)\cap L^p(\Omega)$ and $\zeta\in\mathcal{M}_1$:
\begin{align}
\langle \partial_tu,v\rangle_{H^{-1}\times H^1} + \int_0^\infty \nu(s)\langle\eta^t(s),v\rangle_{H^{-1}\times H^1} ds & = 0,  \label{cp1} \\ 
a_{E,\beta}(u,\xi) + \langle F'(u_1)-F'(u_2),\xi\rangle_{W^{-\beta,2}\times W^{\beta,2}_0} + \alpha \langle\partial_tu,\xi\rangle_{W^{-\beta,2}\times W^{\beta,2}_0} & = \langle\mu,\xi\rangle_{W^{-\beta,2}\times W^{\beta,2}_0},  \label{cp2} \\
(\partial_t\eta^t,\zeta)_{\mathcal{M}_{-1}} - (T_r\eta^t,\zeta)_{\mathcal{M}_{-1}} & = (\mu,\zeta)_{\mathcal{M}_0}  \label{cp3}
\end{align}
with the initial data
\begin{align*}
u(0)=u_{01}-u_{02}, \quad \eta^0=\eta_{01}-\eta_{02}.
\end{align*}
In \eqref{cp1} we choose $v=\mu$ and in \eqref{cp2} we choose $\xi=\partial_tu$.
Owing to Theorem \ref{t:st-solns}, for each $t\in[0,T]$, these elements are in $H^1(\Omega)$ and $W^{\beta,2}_0(\Omega)$, respectively, then sum the results to obtain
\begin{align}
(A_{E,\beta}u,\partial_tu) + (F'(u_1)-F'(u_2),\partial_tu) + \alpha\|\partial_tu\|^2 + \int_0^\infty \nu(s)(\mu,\eta^t(s))ds = 0.  \label{cp3.5}
\end{align}
Further, multiply \eqref{cp3} by $A^{-1}_N\eta^t$ in $\mathcal{M}_0$, then adding the obtained relation to \eqref{cp3.5}, we have 
\begin{align}
\frac{1}{2}\frac{d}{dt}\{ |\|u\||^2_{W^{\beta,2}_0} + \|\eta^t\|^2_{\mathcal{M}_{-1}} \} + \alpha\|\partial_tu\|^2 - \int_0^\infty \nu'(s)\|\eta^t(s)\|^2_{H^{-1}}ds + (F'(u_1)-F'(u_2),\partial_tu) = 0.  \label{cp4}
\end{align}
Using H\"{o}lder's inequality, (N4), Young's inequality and the embedding $L^\infty(\Omega)\hookrightarrow W^{\beta,2}_0(\Omega)$, we estimate the remaining product as 
\begin{align}
|(F'(u_1)-F'(u_2),\partial_tu)| & \le \|F'(u_1)-F'(u_2)\|\|\partial_t u\| \notag \\
& \le C\|(1+|u_1|^{\rho-2}+|u_2|^{\rho-2})u\|\|\partial_tu\|  \notag \\ 
& \le C(1+\|u_1\|^{\rho-2}_{L^{2(\rho-2)}}+\|u_2\|^{\rho-2}_{L^{2(\rho-2)}})\|u\|_{L^\infty}\|\partial_tu\|  \notag \\ 
& \le Q_\alpha(\|(u_{0i},\eta_{0i})\|_{\mathcal{H}^M_{\beta,0}})|\|u\||^2_{W^{\beta,2}_0}+\frac{\alpha}{2}\|\partial_tu\|^2,  \label{cp5}
\end{align}
where the positive monotone increasing function $Q_\alpha(\cdot)\sim\alpha^{-1}$ (we remind the reader $\|(u_{0i},\eta_{0i})\|_{\mathcal{H}^M_{\beta+1,\beta+1}} \le Q\|(u_{0i},\eta_{0i})\|_{\mathcal{H}^M_{\beta,0}}$, for $i=1,2$ and the bounds on $u_1$ and $u_2$ follow from \eqref{wk-1} and \eqref{wk-2}).
With \eqref{cp4} and \eqref{cp5}, we obtain the following differential inequality which holds for almost all $t\in[0,T]$
\begin{align}
\frac{d}{dt}\{ |\|u\||^2_{W^{\beta,2}_0} + \|\eta^t\|^2_{\mathcal{M}_{-1}} \} + \alpha\|\partial_tu\|^2 + \|\eta^t\|^2_{L^2_{-\nu'}(\mathbb{R}_+;H^{-1})} & \le Q_\alpha(\|(u_{0i},\eta_{0i})\|_{\mathcal{H}^M_{\beta,0}})|\|u\||^2_{W^{\beta,2}_0}  \notag \\ 
& \le Q_\alpha(\|(u_{0i},\eta_{0i})\|_{\mathcal{H}^M_{\beta,0}}) \left( |\|u\||^2_{W^{\beta,2}_0} + \|\eta^t\|^2_{\mathcal{M}_{-1}} \right).  \label{cp6}
\end{align}
Employing a Gr\"onwall inequality to \eqref{cp6}, we obtain, for all $t\in[0,T],$
\begin{align}
|\|u(t)\||^2_{W^{\beta,2}_0} + \|\eta^t\|^2_{\mathcal{M}_{-1}} & + \int_0^t \left( \alpha\|\partial_tu(\tau)\|^2 + \|\eta^\tau\|^2_{L^2_{-\nu'}(\mathbb{R}_+;H^{-1})} \right) d\tau  \notag \\ 
& \le e^{C_\alpha} \left( |\|u(0)\||^2_{W^{\beta,2}_0} + \|\eta^0\|^2_{\mathcal{M}_{-1}} \right).  \label{cp7}
\end{align}
This shows the claim \eqref{cp0} holds for the regular solutions.
Since none of the above constants due to the above estimate actually depend on the assumptions of Theorem \ref{t:st-solns}, then standard approximation arguments can be employed to obtain \eqref{cp0} for the weak solutions as well.
\end{proof}

\begin{remark}
It is quite important to remark that in $N=3$ uniqueness for the nonviscous problem (where $\alpha=0$) remains an open problem (indeed, cf. \cite{Conti-Zelati10,GSZ09,GSSZ-09}).
\end{remark}

We now formalize the semi-dynamical system generated by Problem {\bf P}.

\begin{corollary}  \label{t:Lip-semif}
Let the assumptions of Theorem \ref{t:wk-sols} be satisfied. 
Additionally, assume (N4) holds. 
We can define a strongly continuous semigroup of solution operators $\mathcal{S}=(\mathcal{S}(t))_{t\ge0}$, for each $\alpha>0$ and $\beta\in(0,1)$,
\begin{equation*}
\mathcal{S}(t):\mathcal{X}^M_{\beta,0}\rightarrow \mathcal{X}^M_{\beta,0}
\end{equation*}
by setting, for all $t\ge 0,$ 
\begin{equation*}
\mathcal{S}(t)\phi_0:=\phi(t)
\end{equation*}
where $\phi(t)=(u(t),\eta^t)$ is the unique global weak solution to Problem {\textbf P}.
Furthermore, as a consequence of \eqref{cp0}, the semigroup $\mathcal{S}(t):\mathcal{X}^M_{\beta,0}\rightarrow \mathcal{X}^M_{\beta,0}$ is Lipschitz continuous on $\mathcal{X}^M_{\beta,0}$, uniformly in $t$ on compact intervals. 
\end{corollary}

\section{Absorbing sets and global attractors}  \label{s:gattr}

We now give a dissipation estimate for Problem {\bf P} from which we deduce the existence of a bounded absorbing set and an important uniform bound on the solutions of Problem {\textbf P}.
The existence of an absorbing set will also be used later to show that the semigroup of solution operators $\mathcal{S}$ admits a compact global attractor in the metric space $\mathcal{X}^M_{\beta,0}$.

\begin{lemma}  \label{t:diss-1}
Let $\phi_0=(u_0,\eta_0)^{tr}\in\mathcal{H}^M_{\beta,0}=W^{\beta,2}_0(\Omega)\times\mathcal{M}^{(0)}_{-1}$ for $\beta\in(\frac{N}{4},1),$ $N=1,2,3,$ be such that $F(u_0)\in L^1(\Omega).$
Assume (K1), (K3)-(K5) and (N1)-(N3) hold.
Assume $\phi=(u,\eta)^{tr}$ is a weak solution to Problem {\textbf P}.
There are positive constants $\kappa_1$ and $C$, each depending on $\Omega$ but independent of $t$, $\alpha$ and $\phi_0$, such that, for all $t\ge0$, the following holds 
\begin{align}
& \|\phi(t)\|^2_{\mathcal{H}^M_{\beta,0}} + \int_t^{t+1} \alpha\|\partial_tu(\tau)\|^2 d\tau \le Q(\|\phi_0\|_{\mathcal{H}^M_{\beta,0}})e^{-\kappa_1t} + C,  \label{dss-1}
\end{align}
for some monotonically increasing function $Q$ independent of $t$ and $\alpha$.
\end{lemma}

\begin{proof}
The idea of the proof is from \cite{CGG11}.
We give a formal calculation that can be justified by a suitable Faedo-Galerkin approximation based on the proof of Theorem \ref{t:wk-sols} above.
To begin, define the functional, for all $t\ge0,$	 
\begin{align}
\mathcal{Y}(t):=\mathcal{E}(t)+\ep\alpha\|u(t)\|^2-2\ep\int_0^\infty\nu(s)\left( u(t),A^{-1}_N\eta^t(s) \right)ds,  \label{dss-3} 
\end{align}
where $\ep\in(0,\lambda)$ will be chosen sufficiently small later.
From \eqref{p-1}-\eqref{p-3}, we find
\begin{align}
-\frac{d}{dt} & \int_0^\infty\nu(s)(u,A^{-1}_N\eta^t(s))ds  \notag \\ 
& = \|\partial_tu\|^2_{H^{-1}}-\int_0^\infty\nu(s)(u,A^{-1}_N\partial_t\eta^t(s))ds  \notag \\ 
& = \|\partial_tu\|^2_{H^{-1}}-\int_0^\infty\nu'(s)(u,A^{-1}_N\eta^t(s))ds-\int_0^\infty\nu(s)(u,\mu)ds  \notag \\ 
& = \|\partial_tu\|^2_{H^{-1}}-\int_0^\infty\nu'(s)(u,A^{-1}_N\eta^t(s))ds-\frac{\alpha}{2}\frac{d}{dt}\|u\|^2-|\|u\||^2_{W^{\beta,2}_0}-(F'(u),u).  \label{dss-3.5}
\end{align}
Differentiating $\mathcal{Y}$ with respect to $t$ while keeping in mind \eqref{energy}, \eqref{eq9} (without the index $n$) and \eqref{dss-3.5}, we find 
\begin{align}
\frac{d}{dt}\mathcal{Y}+\ep_0\mathcal{Y}-2\int_0^\infty\nu'(s)\|\eta^t(s)\|^2_{H^{-1}}ds=h(t),  \label{dss-3.9}
\end{align}
for $\ep_0\in(0,\ep)$ where
\begin{align}
h(t)= & -2\alpha\|\partial_tu(t)\|^2+2\ep\|\partial_tu(t)\|^2_{H^{-1}}-2\ep\int_0^\infty\nu'(s)(u(t),A^{-1}_N\eta^t(s))ds  \notag \\ 
& - 2\ep_0(F'(u(t))u(t)-F(u(t)),1) - 2(\ep-\ep_0)(F'(u(t)),u(t)) + \ep_0\|\eta^t\|^2_{\mathcal{M}_{-1}}  \notag \\ 
&-(2\ep-\ep_0)|\|u(t)\||^2_{W^{\beta,2}_0}+\ep_0\ep\alpha\|u(t)\|^2-2\ep_0\ep\int_0^\infty\nu(s)(u(t),A^{-1}_N\eta^t(s))ds+\ep_0C.  \label{dss-4}
\end{align} 
From \eqref{Fcons-1} and \eqref{Fcons-2} (with $M=0$) it follows that
\begin{align}
& -2\ep_0(F'(u(t))u(t)-F(u(t)),1) - 2(\ep-\ep_0)(F'(u(t)),u(t))  \notag \\ 
& \le -(\ep-\ep_0)(|F(u)|,1) + \ep_0C|\|u\||^2_{W^{\beta,2}_0}.  \label{dss-4.1}
\end{align}
Next, using assumption (K4) and the embeddings $H^{-1}(\Omega)\hookleftarrow L^2(\Omega)\hookleftarrow W^{\beta,2}_0(\Omega)$, we find 
\begin{align}
-2\ep\int_0^\infty\nu'(s)(u,A^{-1}_N\eta^t(s))ds & = -2\ep\int_0^\infty\nu'(s)(A^{-1/2}_Nu,A^{-1/2}_N\eta^t(s))ds  \notag \\ 
& \le -\ep\int_0^\infty\nu'(s)\left( \frac{1}{\nu_0}|\|u\||^2_{W^{\beta,2}_0}+C\nu_0\|\eta^t(s)\|^2_{H^{-1}} \right)ds  \notag \\
& \le \ep|\|u\||^2_{W^{\beta,2}_0}-\ep C\int_0^\infty\nu'(s)\|\eta^t(s)\|^2_{H^{-1}}ds,  \label{dss-4.2}
\end{align}
and, now with (K3) and \eqref{eq00} (without the index $n$),
\begin{align}
-2\ep_0\ep \int_0^\infty\nu(s)(u,A^{-1}_N\eta^t(s))ds & \le \ep_0\ep C|\|u\||^2_{W^{\beta,2}_0}+\ep_0\ep\|\eta^t\|^2_{\mathcal{M}_{-1}}.  \label{dss-4.3}
\end{align}
Together \eqref{dss-4}-\eqref{dss-4.3} make the following estimate
\begin{align}
h \le & -2\alpha\|\partial_tu\|^2+2\ep\|\partial_tu\|^2_{H^{-1}}-(\ep-\ep_0(1+C+\ep\alpha C))|\|u\||^2_{W^{\beta,2}_0}+2\ep_0\|\eta^t\|^2_{\mathcal{M}_{-1}}  \notag \\ 
& -\ep C \int_0^\infty\nu'(s)\|\eta^t(s)\|^2_{H^{-1}}ds+C.  \label{dss-5}
\end{align}
Here we employ assumption (K5) so that from \eqref{dss-3.9} and \eqref{dss-5} we are able to fix $\ep\in(0,\lambda)$ and $\ep_0\in(0,\ep)$ sufficiently small to, in turn, find positive constants $\ep_1,\ep_2,\ep_3$ so that there holds
\begin{align}
\frac{d}{dt}\mathcal{Y}+\ep_1\mathcal{Y}+2\|\eta^t\|^2_{\mathcal{M}_{-1}}+\ep_2\alpha\|\partial_tu\|^2+\ep_3|\|u\||^2_{W^{\beta,2}_0} \le C.  \label{dss-7}
\end{align}
It is important to note that $C$ on the right-hand side of \eqref{dss-7} is independent of $t$ and $\phi_0.$
One can readily show (cf. \eqref{energy}, \eqref{eq10}-\eqref{eq10.5}) that there holds, for all $t\ge0,$
\begin{align}
C_1\|\phi(t)\|^2_{\mathcal{H}^M_{\beta,0}}-C_2 \le \mathcal{Y}(t) \le Q(\|\phi_0\|_{\mathcal{H}^M_{\beta,0}}),  \label{dss-8}
\end{align}
for some positive constants $C_1,C_2$, and for some monotone nondecreasing function $Q$ independent of $t$.
Finally, by applying a Gr\"onwall type inequality to \eqref{dss-7} (cf. e.g. \cite[Lemma 2.5]{GPV03}), then integrating the result and applying \eqref{dss-8} yield the claim \eqref{dss-1}.
This finishes the proof.
\end{proof}

We immediately deduce the existence of a bounded absorbing set from Lemma \ref{t:diss-1}.

\begin{proposition}
Let the assumptions of Lemma \ref{t:diss-1} hold.
Additionally, assume (N4) holds. 
Then there exists $R_0>0$, independent of $t$ and $\phi_0$, such that $\mathcal{S}(t)$ possesses an absorbing ball $\mathcal{B}^M_{\beta,0}(R_0)\subset\mathcal{H}^M_{\beta,0},$ bounded in $\mathcal{H}^M_{\beta,0}$.
Precisely, for any bounded subset $B\subset\mathcal{H}^M_{\beta,0},$ there exists $t_0=t_0(B)>0$ such that $\mathcal{S}(t)B\subset\mathcal{B}^M_{\beta,0}(R_0)$, for all $t\ge t_0.$
Moreover, for every $R>0$, there exists $C_*=C_*(R)\ge0,$ such that, for any $\phi_0\in\mathcal{B}^M_{\beta,0}(R),$
\begin{align}
\sup_{t\ge0}\|\mathcal{S}(t)\phi_0\|_{\mathcal{H}^M_{\beta,0}}+\int_0^\infty\|\partial_tu(\tau)\|^2d\tau\le C_*,  \label{unif-bnd}
\end{align}
where $\mathcal{B}^M_{\beta,0}(R)$ denotes the ball in $\mathcal{H}^M_{\beta,0}$ of radius $R$, centered at ${\mathbf 0}$.
\end{proposition}

Throughout the remainder of the article, we simply write $\mathcal{B}^M_{\beta,0}$ in place of $\mathcal{B}^M_{\beta,0}(R_0)$ to denote the bounded absorbing set admitted by the semigroup of solution operators $\mathcal{S}(t)$.

For the rest of this section, our aim is to prove the following.

\begin{theorem}  \label{t:global}
Let the assumptions of Lemma \ref{t:diss-1} hold.
Additionally, assume (N4) holds. 
The dynamical system $(\mathcal{X}^M_{\beta,0},\mathcal{S}(t))$ (see Corollary \ref{t:Lip-semif}) possesses a connected global attractor $\mathcal{A}^M_{\beta,0}$ in $\mathcal{H}^M_{\beta,0}.$
Precisely, 

\begin{description}

\item[1] for each $t\geq 0$, $\mathcal{S}(t)\mathcal{A}^M_{\beta,0} = \mathcal{A}^M_{\beta,0}$, and 

\item[2] for every nonempty bounded subset $B$ of $\mathcal{H}_{\beta,0}^M$,
\[
\lim_{t\rightarrow\infty}{\rm{dist}}_{\mathcal{H}_{\beta,0}^M}(\mathcal{S}(t)B,\mathcal{A}^M_{\beta,0}) := \lim_{t\rightarrow\infty}\sup_{\zeta\in B}\inf_{\xi\in\mathcal{A}^M_{\beta,0}}\|\mathcal{S}(t)\zeta-\xi\|_{\mathcal{H}^M_{\beta,0}} = 0.
\]

\end{description}

\noindent Additionally,

\begin{description}

\item[3] the global attractor is the unique maximal compact invariant subset in $\mathcal{H}^M_{\beta,0}$ given by
\[
\mathcal{A}^M_{\beta,0} := \omega (\mathcal{B}^M_{\beta,0}) := \bigcap_{s\geq 0}{\overline{\bigcup_{t\geq s}\mathcal{S}(t)\mathcal{B}^M_{\beta,0}}}^{\mathcal{H}^M_{\beta,0}}.
\]

\end{description}

\noindent Furthermore, 

\begin{description}

\item[4] The global attractor $\mathcal{A}^M_{\beta,0}$ is connected and given by the union of the unstable manifolds connecting the equilibria of $\mathcal{S}(t)$.

\item[5] For each $\zeta_0=(\phi_0,\theta_0)^{tr}\in\mathcal{H}^M_{\beta,0}$, the set $\omega(\zeta_0)$ is a connected compact invariant set, consisting of the fixed points of $\mathcal{S}(t).$

\end{description}

\end{theorem}

With the existence of a bounded absorbing set $\mathcal{B}^M_{\beta,0}$ (in Lemma \ref{t:diss-1}), the existence of a global attractor now depends on the precompactness of the semigroup of solution operators $\mathcal{S}$. 
To this end we will show there is a $t_*>0$ such that the map $\mathcal{S}(t_*)$ is a so-called $\alpha$-contraction on $\mathcal{B}^M_{\beta,0}$; that is, there exist a time $t_*>0$, a constant $0<\kappa<1$ and a precompact pseudometric $M_*$ on $\mathcal{B}^M_{\beta,0}$ such that, for all $\phi_{01},\phi_{02} \in \mathcal{B}^M_{\beta,0}$,
\begin{equation}
\|\mathcal{S}(t_*)\phi_{01} - \mathcal{S}(t_*)\phi_{02}\|_{\mathcal{H}^M_{\beta,0}} \le \kappa\|\phi_{01} - \phi_{02}\|_{\mathcal{H}^M_{\beta,0}} + M_*(\phi_{01},\phi_{02}). \label{pseudo-m}
\end{equation}
Such a contraction is commonly used in connection with phase-field type equations as an alternative to establish the precompactness of a semigroup; for some particular recent results see \cite{Grasselli-2012,Grasselli-Schimperna-2011,Zheng&Milani05}.

\begin{lemma}  \label{t:diff-est}
Under the assumptions of Proposition \ref{t:cont-dep} where $\phi_{01},\phi_{02}\in\mathcal{B}^M_{\beta,0}$, there are positive constants $\kappa_2, C_1$ and $C_{2\alpha}\sim\alpha^{-1}$, each depending on $\Omega$ but independent of $t$ and $\phi_{01},\phi_{02}$, such that, for all $t\ge0,$ 
\begin{align}
\|\phi_1(t)-\phi_2(t)\|^2_{\mathcal{H}^M_{\beta,0}} & \le C_1e^{-\kappa_2t} \|\phi_1(0)-\phi_2(0)\|^2_{\mathcal{H}^M_{\beta,0}} \notag \\
& + C_{2\alpha}\int_0^t \left(\|\nabla\mu_1(\tau)-\nabla\mu_2(\tau)\|^2+\|u_1(\tau)-u_2(\tau)\|^2 \right)d\tau.  \label{pm-0}
\end{align}
\end{lemma}

\begin{proof}
The proof is based on the proof of Proposition \ref{t:cont-dep}.
We begin by recovering \eqref{cp4} by multiplying \eqref{cp1} and \eqref{cp2} by $\mu$ and $\partial_tu$, respectively, in $L^2(\Omega)$, and multiplying \eqref{cp3} by $A^{-1}_N\eta^t$ in $\mathcal{M}_0$, then adding the obtained relations together to find
\begin{align}
\frac{1}{2}\frac{d}{dt}\{ |\|u\||^2_{W^{\beta,2}_0} + \|\eta^t\|^2_{\mathcal{M}_{-1}} \} + \alpha\|\partial_tu\|^2 - \int_0^\infty v'(s)\|\eta^t(s)\|^2_{H^{-1}}ds + (F'(u_1)-F'(u_2),\partial_tu) = 0.  \label{pm-0.3}
\end{align}
Recall $\phi_1=(u_1,\eta_1)$, $\phi_2=(u_2,\eta_2)$ are the unique weak solutions corresponding to the initial data $\phi_{01}$ and $\phi_{02}$, respectively; also, $u=u_1-u_2$ and $\eta^t=\eta^t_1-\eta^t_2$ formally satisfy \eqref{cp1}-\eqref{cp2}.
Applying assumption (K5) and the estimate based on (N4),
\begin{align}
|(F'(u_1)-F'(u_2),\partial_tu)| & \le \|F'(u_1)-F'(u_2)\|\|\partial_t u\| \notag \\
& \le C\|(1+|u_1|^{\rho-2}+|u_2|^{\rho-2})u\|\|\partial_tu\|  \notag \\ 
& \le C(1+\|u_1\|^{\rho-2}_{L^{2(\rho-2)}}+\|u_2\|^{\rho-2}_{L^{2(\rho-2)}})\|u\|_{L^\infty}\|\partial_tu\|  \notag \\ 
& \le Q_\alpha(\|(u_{0i},\eta_{0i})\|_{\mathcal{H}^M_{\beta,0}})|\|u\||^2_{W^{\beta,2}_0}+\frac{\alpha}{2}\|\partial_tu\|^2 \\
& \le Q_\alpha(\|(u_{0i},\eta_{0i})\|_{\mathcal{H}^M_{\beta,0}})+\frac{\alpha}{2}\|\partial_tu\|^2,  \label{pm-0.4}
\end{align}
where the positive monotone increasing function $Q_\alpha(\cdot)\sim\alpha^{-1},$ we find the differential inequality
\begin{align}
\frac{1}{2}\frac{d}{dt}\{ |\|u\||^2_{W^{\beta,2}_0} + \|\eta^t\|^2_{\mathcal{M}_{-1}} \} + \frac{\alpha}{2}\|\partial_tu\|^2 + \lambda \|\eta^t\|^2_{\mathcal{M}_{-1}} \le Q_\alpha(\|(u_{0i},\eta_{0i})\|_{\mathcal{H}^M_{\beta,0}}).  \label{pm-0.5}
\end{align}
In addition, we now multiply \eqref{cp2} by $u$ in $L^2(\Omega)$ to obtain
\begin{align}
|\|u\||^2_{W^{\beta,2}_0}+(F'(u_1)-F'(u_2),u)+\frac{\alpha}{2}\frac{d}{dt}\|u\|^2 = (\mu,u).  \label{pm-1}
\end{align}
Estimating the first product above using (N1) yields 
\begin{align}
(F'(u_1)-F'(u_2),u) & \ge -c_F\|u\|^2.  \label{pm-2}
\end{align}
We also estimate with Young's inequality 
\begin{align}
(\mu,u) & \le \frac{1}{2}\|\mu\|^2 + \frac{1}{2}\|u\|^2.  \label{pm-3}
\end{align}
Combining \eqref{pm-0.5}-\eqref{pm-3} yields 
\begin{align}
\frac{1}{2}\frac{d}{dt} & \left\{ |\|u\||^2_{W^{\beta,2}_0} + \|\eta^t\|^2_{\mathcal{M}_{-1}}+ \frac{\alpha}{2}\|u\|^2 \right\} + \frac{\alpha}{2}\|\partial_tu\|^2 + |\|u\||^2_{W^{\beta,2}_0} + \lambda \|\eta^t\|^2_{\mathcal{M}_{-1}} \notag \\ 
& \le \frac{1}{2}\|\mu\|^2 + Q_\alpha(\|(u_{0i},\eta_{0i})\|_{\mathcal{H}^M_{\beta,0}})|\|u\||^2_{W^{\beta,2}_0}.  \label{pm-4}
\end{align}
Then adding $\frac{\alpha}{2}\|u\|^2$ to each side of \eqref{pm-4} we find, 
\begin{align}
\frac{d}{dt} \mathcal{N} + c\mathcal{N} + \alpha\|\partial_tu\|^2 \le \|\mu\|^2 + Q_\alpha(\|(u_{0i},\eta_{0i})\|_{\mathcal{H}^M_{\beta,0}}),  \label{pm-5}
\end{align}
where $c=\min\{2,2\lambda,\alpha\}$ and 
\begin{align}
\mathcal{N}(t):=|\|u(t)\||^2_{W^{\beta,2}_0} + \|\eta^t\|^2_{\mathcal{M}_{-1}} +  \frac{\alpha}{2}\|u(t)\|^2.
\end{align}
Applying Gr\"{o}nwall's inequality to \eqref{pm-5} after omitting the term $\alpha\|\partial_tu\|^2$, we obtain the claim \eqref{pm-0}.
\end{proof}

Consequently, we deduce the following precompactness result for the semigroup $\mathcal{S}$.

\begin{proposition}  \label{t:pseudometric}
Let the assumptions of Lemma \ref{t:diff-est} hold. 
There is $t_*>0$ such that the operator $\mathcal{S}(t_*)$ is a strict contraction up to the precompact pseudometric on $\mathcal{B}^M_{\beta,0}$, in the sense of \eqref{pseudo-m}, where
\begin{align}
M_*(\phi_{01},\phi_{02}) & := 
C_{2\alpha} \left( \int_0^{t_*} \left(\|\nabla\mu_1(\tau)-\nabla\mu_2(\tau)\|^2+\|u_1(\tau)-u_2(\tau)\|^2 \right) d\tau \right)^{1/2},  \label{pseudo-2}
\end{align}
with $C_\alpha\sim\alpha^{-1}$.
Furthermore, $\mathcal{S}$ is precompact on $\mathcal{B}^M_{\beta,0}$.
\end{proposition}

\begin{proof}
Naturally we follow from the conclusion of Lemma \ref{t:diff-est}.
Clearly there is a $t_*>0$ so that $C_1e^{-\kappa_2t_*/2}<1.$
Thus, the operator $\mathcal{S}(t_*)$ is a strict contraction up to the pseudometric $M_*$ defined by \eqref{pseudo-2}.
The pseudometric $M_*$ is precompact thanks to the Aubin-Lions compact embedding \eqref{compact-u}.
This completes the proof.
\end{proof}

\begin{proof}[Proof of Theorem \ref{t:global}]
The precompactness of the solution operators $\mathcal{S}$ follows via the method of precompact pseudometrics (see Proposition \ref{t:pseudometric}).
With the existence of a bounded absorbing set $\mathcal{B}^M_{\beta,0}$ in $\mathcal{H}^M_{\beta,0}$ (Lemma \ref{t:diss-1}), the existence of a global attractor in $\mathcal{H}^M_{\beta,0}$ is well-known and can be found in \cite{Babin&Vishik92,Temam88} for example.
Additional characteristics of the attractor follow thanks to the gradient structure of Problem {\textbf P} (Remark \ref{r:gradient}).
In particular, the first three claims in the statement of Theorem \ref{t:global} are a direct result of the existence of an absorbing set, a Lyapunov functional $\mathcal{E}$, and the fact that the system $(\mathcal{X}^M_{\beta,0},\mathcal{S}(t),\mathcal{E})$ is gradient. 
The fourth property is a direct result of \cite[Theorem VII.4.1]{Temam88}, and the fifth follows from \cite[Theorem 6.3.2]{Zheng04}.
This concludes the proof.
\end{proof}

\appendix
\section{} \label{s:append-a}

The following is reported from \cite[Theorem 2.5]{Gal-Warma-15-1}.

\begin{theorem}  \label{t:append-1}
Let $0<\beta<1$.
For $K\in\{E,D\},$ the following assertions hold.
\begin{enumerate}
\item[(a)] The operator $-A_{K,\beta}$ generates a submarkovian semigroup $(e^{-A_{K,\beta}})_{t\ge0}$ on $L^2(\Omega)$ and hence can be extended to a strongly continuous contraction semigroup on $L^p(\Omega)$ for every $p\in[1,\infty)$, and to a contraction semigroup on $L^\infty(\Omega).$
\item[(b)] The operator $A_{K,\beta}$ has a compact resolvent, and hence has a discrete spectrum.
The spectrum of $A_{K,\beta}$ may be ordered as an increasing sequence of real numbers $0\le\lambda_1<\lambda_2<\cdots<\lambda_k<\cdots$ that diverges to $+\infty.$
Moreover, $0$ is not an eigenvalue for $A_{K,\beta}$, and if $\phi_k$ is an eigenfunction associated with the eigenvalue $\lambda_k$, then $\phi_k\in D(A_{K,\beta})\cap L^\infty(\Omega)$.
\item[(c)] Denoting the generator of the semigroup on $L^p(\Omega)$ by $A_{K,p}$ so that $A_K=A_{K,2}$, then the spectrum of $A_{K,p}$ is independent of $p$ for every $p\in[1,\infty]$.
\item[(d)] There holds $D(A_{K,\beta})\subset L^\infty(\Omega)$ provided that $N<4\beta$.
Let $p\in(2,\infty)$ and assume that $N<4\beta p/(p-2)$.
Then also $D(A_{K,\beta})\subset L^p(\Omega).$
\end{enumerate}
\end{theorem}

\begin{remark}  \label{r:compact}
From \cite[page 1284, after equation (2.3)]{Gal-Warma-15-1}), we know the following embedding is {\em{compact}}
\begin{align}
W^{\beta,2}_0(\Omega)\hookrightarrow L^p(\Omega) \quad \text{when} \quad 1\le p<\star \quad \text{for} \quad \star=\left\{ \begin{array}{ll} \displaystyle\frac{2N}{N-2\beta} & \text{if} \quad N>2\beta \\ +\infty & \text{if} \quad N=2\beta. \end{array} \right.  \label{compact-1.1}
\end{align}
Also,
\begin{align}
W_0^{\beta,2}(\Omega)\hookrightarrow C^{0,h}({\overline{\Omega}}) \quad & \text{with} \quad h:=\beta-\frac{N}{2} \quad \text{if} \quad N<2\beta \quad \text{and} \quad 2<p<\infty.  \notag
\end{align}
\end{remark}

The following result is the classical Aubin-Lions Lemma, reported here for the reader's convenience (cf. \cite{Lions69}, and, e.g. \cite[Lemma 5.51]{Tanabe79} or \cite[Theorem 3.1.1]{Zheng04}). 

\begin{lemma}  \label{t:Lions}
Let $X,Y,Z$ be Banach spaces where $Z\hookleftarrow Y\hookleftarrow X$ with continuous injections, the second being compact.
Then the following embeddings are compact:
\[
W:=\{ \chi\in L^2(0,T;X), \ \partial_t\chi\in L^2(0,T;Z) \} \hookrightarrow L^2(0,T;Y),
\]
and
\[
W':=\{ \chi\in L^\infty(0,T;X), \ \partial_t\chi\in L^2(0,T;Z) \} \hookrightarrow C([0,T];Y).
\]
\end{lemma}

Here we recall the notion of $\alpha$-contraction and provide the main propositions which guarantee the existence of a global attractor for the semigroup of solution operators $\mathcal{S}(t).$

\begin{definition}
Let $X$ be a Banach space and $\alpha$ be a measure of compactness in $X$ (cf., e.g., \cite[Definition A.1]{Zheng&Milani05}). 
Let $B\subset X$.
A continuous map $T:B\rightarrow B$ is an $\alpha$-contraction on $B$, if there exists a number $q\in(0,1)$ such that for every subset $A\subset B$, $\alpha(T(A)) \le q\alpha(A).$
\end{definition}

\begin{proposition}
Assume that $B\subset X$ is closed and bounded, and that $T:B\rightarrow B$ is an $\alpha$-contraction on $B$. 
Define the semigroup generated by the iterations of $T$, i.e. $S:=(T^n)_{n\in\mathbb{N}}$. 
Then the set
\[
\omega(B):=\bigcap_{n\ge0}{\overline{\bigcup_{m\ge n}T^m(B)}}^X
\]
is compact, invariant, and attracts $B$.
\end{proposition}

\begin{proposition}
Assume that $S$ is a continuous semigroup of operators on $X$ admitting a bounded, positively invariant absorbing set $B$, and that there exists $t_*>0$ such that the operator $S_* := S(t_*)$ is an $\alpha$-contraction on $B$. 
Let
\[
A_*:=\bigcap_{n\ge0}{\overline{\bigcup_{m\ge n}S^m_*(B)}}^X=\omega_*(B)
\]
be the $\omega$-limit set of $B$ under the map $S_*$, and set 
\[
A:=\bigcup_{0\le t\le t_*}S(t)A_*.
\]
Assume further that for all $t\in [0,t_*]$, the map $x\rightarrow S(t)x$ is Lipschitz continuous from $B$ to $B$, with Lipschitz constant $L(t)$, $L:[0,t_*]\rightarrow(0,+\infty)$ being a bounded function. 
Then $A=\omega(B)$, and this set is the global attractor of $S$ in $B$.
\end{proposition}

Theorems 3.1 and 3.2 are motivated by \cite[Sections II.2 and III.2]{Hale88}, but appear in the above form in \cite[Appendix A]{Zheng&Milani05} and \cite[Sections II.7]{Milani&Koksch05}.
We also rely on the following.

\begin{definition}
A pseudometric $d$ in $X$ is precompact in $X$ if every bounded sequence has a subsequence which is a Cauchy sequence relative to $d$.
\end{definition}

\begin{proposition}
Let $B\subset X$ be bounded, let $d$ be a precompact pseudometric in $X$, and let $T : B \rightarrow B$ be a continuous map. 
Suppose $T$ satisfies the estimate
\[
\|Tx-Ty\|_X \le q \|x-y\|_X + d(x,y)
\]
for all $x,y\in B$ and some $q\in(0,1)$ independent of $x$ and $y$. 
Then $T$ is an $\alpha$-contraction.
\end{proposition}

\section*{Acknowledgements}

The authors are indebted to the generosity of Professor Ciprian G. Gal whose enthusiasm and insight into this project proved indispensable and enlightening.

\bibliographystyle{amsplain}

\begin{thebibliography}{10}

\bibitem{Adams-Hedberg-96}
David Adams and Lars~Inge Hedberg, \emph{Function spaces and potential theory},
  Grundlehren der Mathematischen Wissenschaften, vol. 314, Springer--Verlag,
  Berlin, 1996.

\bibitem{GA-GS-AS_2015}
Goro Akagi, Giulio Schimperna, and Antonio Segatti, \emph{Fractional
  {C}ahn--{H}illiard, {A}llen--{C}ahn and porous medium equations}, JDE
  \textbf{261} (2016), no.~6, 2935--2985.

\bibitem{AVMRTM10}
Fuensanta Andreu-Vaillo, Jos\'{e}~M. Maz\'{o}n, Julio~D. Rossi, and
  J.~Juli\'{a}n Toledo-Melero, \emph{Nonlocal diffusion problems}, Mathematical
  Surveys and Monographs, vol. 165, American Mathematical Society, Real
  Sociedad Matem\'{a}tica Espa\~{n}ola, 2010.

\bibitem{Babin&Vishik92}
A.~V. Babin and M.~I. Vishik, \emph{Attractors of evolution equations},
  North-Holland, Amsterdam, 1992.

\bibitem{Barbu76}
Viorel Barbu, \emph{Nonlinear semigroups and differential equations in {B}anach
  spaces}, Noordhoff International Publishing, Bucharest, 1976.

\bibitem{Bates&Han05}
Peter Bates and Jianlong Han, \emph{The {N}eumann boundary problem for a
  nonlocal {C}ahn--{H}illiard equation}, J. Differential Equations \textbf{212}
  (2005), no.~2, 235--277.

\bibitem{BBC-03}
Krzysztof Bogdan and Krzysztof Burdzy Zhen-Qing Chen, \emph{Censored stable
  processes}, Probab. Theory Related Fields \textbf{127} (2003), no.~1,
  89--152.

\bibitem{Bucur-Valdinoci-16}
Claudia Bucur and Enrico Valdinoci, \emph{Nonlocal diffusion and applications},
  Lecture Notes of the Unione Matematica Italiana, vol.~20, Springer, Bologna,
  2016.

\bibitem{CRS-10}
Luis~A. Caffarelli, Jean-Michel Roquejoffre, and Yannick Sire,
  \emph{Variational problems with free boundaries for the fractional
  laplacian}, J. Eur. Math. Soc. \textbf{12} (2010), 1151--1179.

\bibitem{Cahn-Hilliard}
J.~W. Cahn and J.~E. Hilliard, \emph{Free energy of a nonuniform system i.
  interfacial energy}, J. Chem. Phys. \textbf{28} (1958), 258--267.

\bibitem{CGG11}
Cecilia Cavaterra, Ciprian Gal, and Maurizio Grasselli,
  \emph{{C}ahn--{H}illiard equations with memory and dynamic boundary
  conditions}, Asymptot. Anal. \textbf{71} (2011), no.~3, 123--162.

\bibitem{CGW-14}
Cecilia Cavaterra, Maurizio Grasselli, and Hao Wu, \emph{Non-isothermal viscous
  {C}ahn--{H}illiard equation with inertial term and dynamic boundary
  conditions}, Commun. Pure Appl. Anal. \textbf{13} (2014), no.~5, 1855--1890.

\bibitem{CFG12}
Pierluigi Colli, Sergio Frigeri, and Maurizio Grasselli, \emph{Global existence
  of weak solutions to a nonlocal {C}ahn--{H}illiard--{N}avier--{S}tokes
  system}, J. Math. Anal. Appl. \textbf{386} (2012), no.~1, 428--444.

\bibitem{CKRS07}
Pierluigi Colli, Pavel Krej\v{c}\'{i}, Elisabetta Rocca, and J\"{u}rgen
  Sprekels, \emph{Nonlinear evolution inclusions arising from phase change
  models}, Czechoslovak Math. J. \textbf{57} (2007), no.~132, 1067--1098.

\bibitem{Conti-Zelati10}
M.~Conti and M.~Coti Zelati, \emph{Attractors for the non-viscous
  {C}ahn--{H}illiard equation with memory in 2{D}}, Nonlinear Anal. \textbf{72}
  (2010), no.~1668--1682.

\bibitem{Conti-Mola-08}
Monica Conti and Gianluca Mola, \emph{3-{D} viscous {C}ahn--{H}illiard equation
  with memory}, Math. Models Methods Appl. Sci. \textbf{32} (2008), no.~11,
  1370--1395.

\bibitem{CPS06}
Monica Conti, Vittorino Pata, and Marco Squassina, \emph{Singular limit of
  differential systems with memory}, Indiana Univ. Math. J. \textbf{55} (2007),
  no.~1, 169--215.

\bibitem{Frigeri&Grasselli12}
Sergio Frigeri and Maurizio Grasselli, \emph{Global and trajectory attractors
  for a nonlocal {C}ahn--{H}illiard--{N}avier--{S}tokes system}, J. Dynam.
  Differential Equations \textbf{24} (2012), no.~4, 827--856.

\bibitem{Frigeri-Grasselli-Rocca_2014}
Sergio Frigeri, Maurizio Grasselli, and Elisabetta Rocca, \emph{A diffuse
  interface model for two-phase incompressible flows with non-local
  interactions and non-constant mobility}, Nonlinearity \textbf{28} (2015),
  no.~5, 1257--1293.

\bibitem{Gal-16-doubly}
Ciprian~G. Gal, \emph{Doubly nonlocal {C}ahn--{H}illiard equations}, Annales de
  l'Institut Henri Poincar\'{e} (C) Analyse Non Lin\'{e}aire (2016).

\bibitem{Gal-NCHEFDBC-16}
\bysame, \emph{Non-local {C}ahn--{H}illiard equations with fractional dynamic
  boundary conditions}, Euro. Jnl of Applied Mathematics (2016), 53 pp.

\bibitem{Gal-17-str-str-CHE}
\bysame, \emph{On the strong-to-strong interaction case for doubly nonlocal
  {C}ahn--{H}illiard equations}, Discrete Contin. Dyn. Syst. \textbf{37}
  (2017), no.~1, 131--167.

\bibitem{Gal&Grasselli12}
Ciprian~G. Gal and M.~Grasselli, \emph{Singular limit of viscous
  {C}ahn--{H}illiard equations with memory and dynamic boundary conditions},
  DCDS-B \textbf{18} (2013), no.~6, 1581--1610.

\bibitem{Gal&Grasselli14}
Ciprian~G. Gal and Murizio Grasselli, \emph{Longtime behavior of nonlocal
  {C}ahn--{H}illiard equations}, Discrete Contin. Dyn. Syst. \textbf{34}
  (2014), no.~1, 145--179.

\bibitem{Gal&Miranville09}
Ciprian~G. Gal and Alain Miranville, \emph{Uniform global attractors for
  non-isothermal viscous and non-viscous {C}ahn--{H}illiard equations with
  dynamic boundary conditions}, Nonlinear Anal. Real World Appl. \textbf{10}
  (2009), no.~3, 1738--1766.

\bibitem{Gal-Shomberg15-2}
Ciprian~G. Gal and Joseph~L. Shomberg, \emph{Coleman-{G}urtin type equations
  with dynamic boundary conditions}, Phys. D \textbf{292/293} (2015), 29--45.

\bibitem{Gal-Warma-15-1}
Ciprian~G. Gal and Mahamadi Warma, \emph{Reaction-diffusion equations with
  fractional diffusion on non-smooth domains with various boundary conditions},
  Discrete Contin. Dyn. Syst. \textbf{36} (2016), no.~3, 1279--1319.

\bibitem{GGMP05-mem}
S.~Gatti, M.~Grasselli, A.~Miranville, and V.~Pata, \emph{Memory relaxation of
  first order evolution equations}, Nonlinearity \textbf{18} (2005), no.~4,
  1859--1883.

\bibitem{GGPS05}
S.~Gatti, M.~Grasselli, V.~Pata, and M.~Squassina, \emph{Robust exponential
  attractors for a family of nonconserved phase-field systems with memory},
  Discrete Contin. Dyn. Syst. \textbf{12} (2005), no.~5, 1019--1029.

\bibitem{GMPZ10}
S.~Gatti, A.~Miranville, V.~Pata, and S.~Zelik, \emph{Continuous families of
  exponential attractors for singularly perturbed equations with memory}, Proc.
  Roy. Soc. Edinburgh Sect. A \textbf{140} (2010), 329--366.

\bibitem{Giacomin-Lebowitz-97}
Giambattista Giacomin and Joel~L. Lebowitz, \emph{Phase segregation dynamics in
  particle systems with long range interactions. i. macroscopic limits}, J.
  Statist. Phys. \textbf{87} (1997), no.~1--2, 37--61.

\bibitem{GGP01}
C.~Giorgi, M.~Grasselli, and V.~Pata, \emph{Well-posedness and longtime
  behavior of the phase-field model with memory in a history space setting},
  Quart. Appl. Math. \textbf{59} (2001), 701--736.

\bibitem{GPM98}
Claudio Giorgi, Vittorino Pata, and Alfredo Marzocchi, \emph{Asymptotic
  behavior of a semilinear problem in heat conduction with memory}, NoDEA
  Nonlinear Differential Equations Appl. \textbf{5} (1998), no.~3, 333--354.

\bibitem{GPM00}
\bysame, \emph{Uniform attractors for a non-autonomous semilinear heat equation
  with memory}, Quart. Appl. Math. \textbf{58} (2000), no.~4, 661--683.

\bibitem{Grasselli08}
M.~Grasselli, \emph{On the large time behavior of a phase-field system with
  memory}, Asymptot. Anal. \textbf{56} (2008), 229--249.

\bibitem{GSZ09}
M.~Grasselli, G.~Schimperna, and S.~Zelik, \emph{On the 2{D} {C}ahn--{H}illiard
  equation with inertial term}, Comm. Partial Differential Equations
  \textbf{34} (2009), 707--737.

\bibitem{Grasselli-2012}
Maurizio Grasselli, \emph{Finite-dimensional global attractor for a nonlocal
  phase-field system}, Istituto Lombardo (Rend. Scienze) Mathematica
  \textbf{146} (2012), 113--132.

\bibitem{Grasselli&Pata02-2}
Maurizio Grasselli and Vittorino Pata, \emph{Uniform attractors of
  nonautonomous dynamical systems with memory}, Evolution equations, semigroups
  and functional analysis, vol.~50, Birkh\"{a}user, Boston, MA, 2002.

\bibitem{Grasselli&Pata05}
\bysame, \emph{Robust exponential attractors for a phase-field system with
  memory}, J. Evol. Equ. \textbf{5} (2005), no.~4, 465--483.

\bibitem{GPV03}
Maurizio Grasselli, Vittorino Pata, and Federico Vegni, \emph{Longterm dynamics
  of a conserved phase-field system with memory}, Asymptot. Anal. \textbf{33}
  (2003), no.~3-4, 261--320.

\bibitem{GPS07}
Maurizio Grasselli, Hana Petzeltov\'{a}, and Giulio Schimperna,
  \emph{Asymptotic behavior of a nonisothermal viscous cahn-hilliard equation
  with inertial term}, J. Differential Equations \textbf{239} (2007), no.~1,
  38--60.

\bibitem{GSSZ-09}
Maurizio Grasselli, Giulio Schimperna, Antonio Segatti, and Sergey Zelik,
  \emph{On the {3D} {C}ahn--{H}illiard equation with inertial term}, J. Evol.
  Equ. \textbf{9} (2009), 371--404.

\bibitem{Grasselli-Schimperna-2011}
Murizio Grasselli and Giulio Schimperna, \emph{Nonlocal phase-field systems
  with general potentials}, Discrete Contin. Dyn. Syst. \textbf{33} (2013),
  no.~11--12, 5089--5106.

\bibitem{Hale88}
Jack~K. Hale, \emph{Asymptotic behavior of dissipative systems}, Mathematical
  Surveys and Monographs - No. 25, American Mathematical Society, Providence,
  1988.

\bibitem{Lions69}
J.~L. Lions, \emph{Quelques m\'ethodes de r\'esolution des probl\`emes aux
  limites non lin\'eaires}, Dunod, Paris, 1969.

\bibitem{Milani&Koksch05}
Albert~J. Milani and Norbert~J. Koksch, \emph{An introduction to semiflows},
  Monographs and Surveys in Pure and Applied Mathematics - Volume 134, Chapman
  \& Hall/CRC, Boca Raton, 2005.

\bibitem{Pata-Zucchi-2001}
Vittorino Pata and Adele Zucchi, \emph{Attractors for a damped hyperbolic
  equation with linear memory}, Adv. Math. Sci. Appl. \textbf{11} (2001),
  no.~2, 505--529.

\bibitem{Pazy83}
Amnon Pazy, \emph{Semigroups of linear operators and applications to partial
  differential equations}, Applied Mathematical Sciences - Volume 44,
  Springer-Verlag, New York, 1983.

\bibitem{Porta-Grasselli-2014}
Francesco~Della Porta and Maurizio Grasselli, \emph{Convective nonlocal
  cahn-hilliard equations with reaction terms}, Discrete Contin. Dyn. Syst.
  Ser. B \textbf{20} (2015), no.~5, 1529--1553.

\bibitem{SS-16}
Riccardo Scala and Giulio Schimperna, \emph{On the viscous {C}ahn--{H}illiard
  equation with singular potential and inertial term}, AIMS Mathematics:
  Nonlinear Evolution PDEs, Interfaces and Applications \textbf{1} (2016),
  no.~1, 64--76.

\bibitem{Shomberg18-1}
Joseph~L. Shomberg, \emph{Upper-semicontinuity of the global attractors for a
  class of nonlocal {C}ahn--{H}illiard equations}, submitted.

\bibitem{Shomberg-reacg16}
\bysame, \emph{Robust exponential attractors for {C}oleman--gurtin equations
  with dynamic boundary conditions possessing memory}, Electron. J.
  Differential Equations \textbf{2016} (2016), no.~47, 1--35.

\bibitem{Shomberg-n16}
\bysame, \emph{Well-posedness and global attractors for a non-isothermal
  viscous relaxation of nonlocal {C}ahn--{H}illiard equations}, AIMS
  Mathematics: Nonlinear Evolution PDEs, Interfaces and Applications \textbf{1}
  (2016), no.~2, 102--136.

\bibitem{Tanabe79}
Hiroki Tanabe, \emph{Equations of evolution}, Pitman, London, 1979.

\bibitem{Temam88}
Roger Temam, \emph{Infinite-dimensional dynamical systems in mechanics and
  physics}, Applied Mathematical Sciences - Volume 68, Springer-Verlag, New
  York, 1988.

\bibitem{Temam01}
\bysame, \emph{{N}avier-{S}tokes equations - theory and numerical analysis},
  reprint ed., AMS Chelsea Publishing, Providence, 2001.

\bibitem{Warma15}
Mahamadi Warma, \emph{The fractional relative capacity and the fractional
  {L}aplacian with {N}eumann and {R}obin boundary conditions on open sets},
  Potential Analysis (to appear), DOI: 10.1007/s11118--014--9443--4.

\bibitem{Zheng04}
Songmu Zheng, \emph{Nonlinear evolution equations}, Monographs and Surveys in
  Pure and Applied Mathematics - Volume 133, Chapman \& Hall/CRC, Boca Raton,
  2004.

\bibitem{Zheng&Milani05}
Songmu Zheng and Albert Milani, \emph{Global attractors for singular
  perturbations of the {C}ahn--{H}illiard equations}, J. Differential Equations
  \textbf{209} (2005), no.~1, 101--139.

\end{thebibliography}
\providecommand{\bysame}{\leavevmode\hbox to3em{\hrulefill}\thinspace}
\providecommand{\MR}{\relax\ifhmode\unskip\space\fi MR }
\providecommand{\MRhref}[2]{%
  \href{http://www.ams.org/mathscinet-getitem?mr=#1}{#2}
}
\providecommand{\href}[2]{#2}

\end{document}